\documentclass[12pt,a4paper]{amsart}
\usepackage[dvipsnames]{xcolor}
\usepackage{amsmath,amsfonts,amsthm,amssymb,graphics,graphicx,dsfont,verbatim,quiver,mathtools,adjustbox,cases,float}
\usepackage{faktor,mathrsfs,mathabx,bm,relsize}

\newtheorem{thm}{Theorem}[section]
\newtheorem{defin}[thm]{Definition}

\newtheorem{lem}[thm]{Lemma}

\newtheorem{prop}[thm]{Proposition}

\theoremstyle{remark} \newtheorem{remark}[thm]{Remark}

\newcommand{\Tr}{\text{Trace}}
\newcommand{\ee}{\boldsymbol{e}}

\newcommand{\R}{\mathbb{R}}
 \newcommand{\F}{\boldsymbol{F}}
 \newcommand{\G}{\boldsymbol{G}}
\DeclareRobustCommand{\e}[1]{\underleftrightarrow{\eta_{#1}}}
\newcommand{\bsigma}{\boldsymbol{\sigma}}

 \newcommand{\C}{\mathbb{C}}
\newcommand{\N}{\mathbb{N}}

\newcommand{\B}{\mathbb{B}}
\renewcommand{\P}{\mathbb{P}}
\newcommand{\T}{\boldsymbol{T}}
\newcommand{\bdelta}{\boldsymbol{\delta}}
\newcommand\Id{\operatorname{Id}}
\newcommand{\U}{{\boldsymbol{U}}}
\newcommand{\bmu}{\boldsymbol{\mu}}
\newcommand{\norm}[1]{\left|\!\left|{#1}\right|\!\right|}
\newcommand{\1}{\ensuremath{\mathds{1}}}
\newcommand{\btau}{\boldsymbol{\tau}}
 \newcommand{\CC}{\mathfrak{C}}
\newcommand{\weave}{\mathscr{W}}
\renewcommand{\SS}{\mathfrak{S}}

\newcommand{\simga}{\sigma}

\title[Simultaneous saturation]{The principle of simultaneous saturation: Application to the $k$-linear restriction/extension problem}
\author{Melissa Tacy}

\begin{document}

\begin{abstract}

This paper develops a new framework, \emph{simultaneous saturation}, designed to quantify the size of sets whose elements are simultaneously large.  The framework establishes a correspondence between the magnitude of such sets and a system of interdependent conditions linking their points. We first prove a general theorem establishing the correspondence and then apply the framework to multilinear restriction-type estimates.  From this perspective, we obtain a new proof (independent of Bennett-Carbery-Tao \cite{BCT}) of the $d$-linear restriction/extension theorem, and establish the $\lambda^{\epsilon}$ loss conjectured bounds for the $k$-linear $L^{2}\to L^{p/k}$ extension problem under mixed transversality/curvature conditions $(k<d)$.

\end{abstract}
\maketitle

Analysts frequently wish to control the size of objects, or families of objects.  For instance, to study the mapping properties of an operator $T:L^{q}(X)\to L^{p}(Y)$, one seeks to understand the size of the image set
\[\{Tf(p)\mid p\in Y, f\in L^{q}(X)\}.\]
In many cases it is relatively straightforward to determine a sharp supremum bound describing how large $|Tf(p)|$ can be at a single point. Such an estimate is said to be \emph{sharp} if it can be attained. That is if there exist $f$ and $p$ such that $|Tf(p)|=B$.  In parameter-dependent settings one often weakens the notion of saturation a little, requiring only that $|Tf(p)|\ge cB$ for some universal constant $c>0$.

Typically, many pairs $(f,p)$ can saturate the bound, and there is an interdependence between the function $f$ and the point $p$.  To obtain a more refined understanding of the \emph{saturation set}
\[\{Tf(p)\mid (f,p)\text{ saturates}\},\]
one must ask whether simultaneous saturation can occur. That is, whether a single function $f$ can satisfy $|Tf(p_{j})|=B$ at many distinct points $(p_{1},\dots,p_{N})$.

If such simultaneous saturation occurs, then any average of the values $\{|Tf(p_{1})|,\dots,|Tf(p_{N})|\}$ must also attain the $L^{\infty}$ bound.  The simplest choice is to consider the arithmetic average,
\[\frac{1}{N}\sum_{j=1}^{N}\alpha_{j}Tf(p_{j}), \qquad |\alpha_{j}|=1.\]
One might then attempt to show that this average is necessarily smaller than the saturation bound, perhaps through cancellations among the terms.  This approach, however, has limitations.  If some of the values $|Tf(p_{j})|$ are large while others are small, the average may still be large; arithmetic averaging is not sensitive to the presence of small data.  Moreover, exploiting cancellations in such sums is in itself a delicate and often intractable problem.

Rather than considering arithmetic averages, we instead propose to study geometric averages,
\[\left(\prod_{j=1}^{N}|Tf(p_{j})|\right)^{1/N}.\]
Geometric averages are significantly more sensitive to small values: even a single small term reduces the product substantially. A zero value obliterates it.

The guiding principle in the simultaneous saturation framework is to establish a quantitative barrier on the number of points at which simultaneous saturation can occur (i.e. derive an upper bound on $N$).  To this end, we consider the product
\begin{equation}\label{prodp}
\prod_{j=1}^{N}Tf(p_{j}),
\end{equation}
which may be viewed as the evaluation of the tensor product operator
\[
\T=\underbrace{T\otimes\cdots\otimes T}_{N\text{ times}}
\quad\text{acting on}\quad
\F=\underbrace{f\otimes\cdots\otimes f}_{N\text{ times}}\]
at the point $P=(p_{1},\dots,p_{N})$.  Since the product \eqref{prodp} is symmetric in the indices, we may freely permute the ordering of the points.  Accordingly, we construct a finite measure space generated by point masses supported on the permutations of $P$, denoted $P_{\sigma}=(p_{\sigma(1)},\dots,p_{\sigma(N)}).$

Within this framework, we need to record whether two pairs of points can be made simultaneously large. This is achieved via an energy matrix. The $(\sigma,\tau)$ entry of the energy matrix gives us information about whether $\T\F$ can be simultaneously large at $P_{\sigma}$ and $P_{\tau}$. In this paper we construct the energy matrix from the operator $\T\T^{\star}$. From the energy matrix, we will show that if N is too large, we would be able to create some illegal ensembles. The ensembles are created by considering paths through the entries occurring in the matrix. For instance we might connect $P_{\sigma_{1}}$ to $P_{\sigma_{100}}$ (the $(\sigma_{1},\sigma_{100})$ entry of the matrix tells us if this is possible) then $P_{\sigma_{100}}$ to $P_{\sigma_{54}}$ (the $(\sigma_{100},\sigma_{54})$ entry of the matrix determines whether this connection is possible). If both connections are possible we have an ensemble 
$$P_{\sigma_{1}}\xleftrightarrow {} P_{\sigma_{100}} \xleftrightarrow{} P_{\sigma_{54}}.$$
For this ensemble to be possible we require that both the $(\sigma_{1},\sigma_{100})$ and the $(\sigma_{100},\sigma_{54})$ entries of the matrix be large. The illegality of an ensemble arises from the properties of the operator $T$ and the points appearing in the ensemble. 

In this paper the illegal ensembles will be loops of points that are illegal due to geometric constraints. To fully demonstrate the technique it is, at this point, helpful to specialise to a particular example.  Suppose for example we consider (as we will do in Section \ref{sec:ML}) the multilinear restriction problem. This problem is closely related to the celebrated Fourier restriction/extension problem, see \cite{Ben14} for a good discussion of the relationship. While the Fourier restriction/extension problem is stated in terms of curvature, the multilinear restriction problem is about transversality. In this problem we consider the multilinear operator
\[\prod_{m=1}^{d}T_{m}f_{m}\quad f_{m}\in L^{2}(H_{m})\]
where
\[T_{m}f(x)=\int_{H_{m}}e^{i\lambda\langle x,\xi\rangle}f(\xi)d\mu_{m}(\xi)\quad x\in\R^{d}\]
where $H_{m}\subset\R^{d}$ is a hypersurface and $\mu_{m}$ its hypersurface measure. The set of hypersurfaces $\{H_{m}\}$ are assumed to have normals $\nu_{m}$ so that
\[|\nu_{1}\wedge\cdots\wedge\nu_{d}(x)|\geq{}C.\]
Since we are studying a multilinear problem, we want to estimate
\[\norm{\prod_{m=1}^{d}T_{m}f_{m}}_{L^{p/d}}\]
in terms of
\[\prod_{m=1}^{d}\norm{f}_{L^{2}}.\]

For the multilinear restriction estimate, we build our illegal ensembles out points connected by directions. First we need the energy matrix. In this case, the energy matrix records information on whether points can share $L^{2}$ energy. For the multilinear restriction problem, it is well known what this relationship (for a pair of points) is. If $\e{m}(x,y)$ is the kernel of $T_{m}T_{m}^{\star}$ simple non-stationary phase results tell us that for points $p_{j_{1}}$ and $p_{j_{2}}$ to share energy ($|\e{m}(p_{j_{1}},p_{j_{2}})|$ is large) it must be the case that
\begin{equation}\frac{p_{j_{1}}-p_{j_{2}}}{|p_{j_{1}}-p_{j_{2}}|}\approx\nu_{m}.\label{pwshare}\end{equation}
If \eqref{pwshare} does not hold then $\e{m}(x,y)$ rapidly decays. That is
\[|\e{m}(x,y)|\leq C_{R}\lambda^{-R}\text{ for any }R\in\N.\]
The energy matrix contains all the information about pointwise energy sharing for each value of $(m,j)$. 

Consider now just the three dimensional case. Suppose we can find three points $\{p_{j_{1}},p_{j_{2}},p_{j_{3}}\}$ so that
\[p_{j_{1}}\xleftrightarrow[\nu_{1}]{}p_{j_{2}}\xleftrightarrow[\nu_{2}]{}p_{j_{3}}\xleftrightarrow[\nu_{3}]{}p_{j_{1}}\]
where by $\xleftrightarrow[\nu_{m}]{}$ we mean that \eqref{pwshare} holds for $\nu_{m}$. If $\{\nu_{1},\nu_{2},\nu_{3}\}$ form a spanning set for $\R^{3}$ (as they must from the transversality condition) we have constructed an illegal ensemble. Why? Three points sitting in $\R^{3}$ must lie in a plane, so they cannot be connected in a loop by vectors that give a spanning set for $\R^{3}$. This constraint can be easily generalised to higher dimensions by constructing loops involving $d$ points and the $d$ coordinate directions. Of course if all the points happen to be the same we can no longer talk sensibly about the direction between them so the trivial loop is  always legal.

This heuristic guides the numerology multilinear restriction. There are $N^{d}$ different ways to construct loops out of $N$ points, only $N$ of those are trivial. So if we average over all loops, but only the trivial loops are legal (and therefore they are the only ones that can contribute) we expect to see an $N^{-(d-1)}$ improvement in the $TT^{\star}$ bound. Therefore an $N^{-\frac{d-1}{2}}$ improvement in the bound on $T$, matching the critical $p=\frac{2}{d-1}$ numerology of the multilinear restriction estimates. 

Here we connect points by directions, but other connections are possible. The results of Section \ref{sec:simsat} provide a quite general framework for structuring a simultaneous saturation argument in terms illegal loops of conditions.

It is worth noting that taking geometric averages to measure simultaneous saturation is effective in other (than evaluation of a function at a point) situations For example suppose $T:X\to Y$ and $T$ can be written as
\[T=\sum_{i=1}^{I}T_{i}\]
where sharp bounds are are known for each of the $\norm{T_{i}}_{X\to Y}$ we may always use the triangle inequality to obtain a bound for $\norm{T}_{X\to Y}$. However in most interesting problems the bound arising from the triangle inequality is not sharp. The principle of simultaneous saturation comes into play when the component pieces are not expected to achieve their sharp norm on acting on the same elements. That is in cases where if $v_{i}\in X$ is such that
\[\norm{T_{i}v_{i}}_{Y}\geq{}c\norm{v_{i}}_{X}\]
then $\norm{T_{k}v_{i}}_{Y}$ is small for all other $k\neq i$.  In such cases we would expect that the operator norm of $T$ would be much smaller than suggested by the triangle inequality because even though each component piece can saturate their bounds they require different inputs to do so.  An extreme form of this can be seen in the setting of almost orthogonality  where $X$ and $Y$ are Hilbert spaces and the $T_{i}$ are almost orthogonal operators. The Cotlar-Stein almost orthogonality theorem is proved by using a tensor power trick taking repeated compositions of $TT^{\star}$, expanding the sum and showing that strings
\[T_{i_{1}}T^{\star}_{i_{2}}\cdots T_{i_{N-1}}T^{\star}_{i_{N}}\]
where some of the $i_{k}$ are distinct decay leaving only the diagonal terms. The reason that this (and similar)  applications of the tensor power trick is that they consider geometric averages and exploit a lack of simultaneous saturation.

The paper is arranged in the following fashion. First in Section \ref{sec:simsat} we prove a general form of the principle of simultaneous saturation, Theorem \ref{thm:simsat}. In Section \ref{sec:ML} we see how this can be used to re-prove the $d$-multilinear restriction estimates of Bennett-Carbery-Tao \cite{BCT} and establish the conjectured near optimal $k$-multilinear estimates under a mixed transversality/curvature condition. 

\section{The principle of simultaneous saturation}\label{sec:simsat} 

In this section we state a very general form of the principle of simultaneous saturation. For $m=1,\dots,M$ let $T_{m}:\mathcal{H}_{m}\to X^{\C}$ where $\mathcal{H}$ is a Hilbert space and  $X$ is a non-empty set (in our motivating example $\mathcal{H}_{m}=L^{2}(H_{m})$ and $X\subset\R^{d}$). For $p\in X$ define the operator 
\[T_{m}(p)f=T_{m}f\Big|_{p}\]
which we view as a map $\mathcal{H}\to \C$ (we equip $\C$ with its usual inner product to make it a Hilbert space).  Therefore the dual map $\left(T_{m}(p)\right)^{\star}:\C\to \mathcal{H}$ and for any $p,q\in X$, $T_{m}(p)\left(T_{m}(q)\right)^{\star}:\C\to \C$ and so is given by a multiplication operator
\[T_{m}(p)\left(T_{m}(q)\right)^{\star}\alpha= \beta_{p,q}\alpha\quad\beta_{p,q}\in\C\]
we will frequently abuse notation and associate $T_{m}(p)\left(T_{m}(q)\right)^{\star}$ with the complex number $\beta_{p,q}$.

\begin{defin}\label{def:system}
Let $\lambda\geq{}1$, We say that a set of operators $\mathbb{T}_{\lambda}=\{T_{1},\cdots, T_{M}\}$ is a non-degenerate, uniformly bounded simultaneous system if the following conditions hold:
\begin{enumerate}
\item[{[1]}] Trivial estimate: There is an bound $B=B(\lambda)\in\R^{+}$ so that for any $p,q\in X$,
\begin{equation}\left|T_{m}(p)(T_{m}(q))^{\star}\right|\lesssim B^{2}\label{trivest}\end{equation} 
\item[{[2]}] Decay conditions: For each $m=1,\dots,M$ there is a symmetric, reflexive relation $\CC_{m}$ on $X$ and a decay rate $\mathfrak{R}(\lambda)$ so that either $p\CC_{m} q$ or $T_{m}(p)\left(T_{m}(q)\right)^{\star}$  decays controlled by $\mathfrak{R}(\lambda)$, that is
\[|T_{m}(p)(T_{m}(q))^{\star}|\leq \mathfrak{R}(\lambda)\]

\item[{[3]}] Non-degeneracy:  If we write $p\CC_{m} q$ diagrammatically as
\[p\xleftrightarrow[\CC_{m}]{}q\]
then for any $n\leq M$ and each $\CC_{m_{i}}$ distinct the diagram

\[\begin{tikzcd}
	& {p_{2}} && { p_{3}} \\
	{ p_{1}} &&&& { p_{4}} \\
	& {p_{n}} & \cdots & {}
	\arrow["{\mathfrak{C}_{m_{2}}}", tail reversed, from=1-2, to=1-4]
	\arrow["{\mathfrak{C}_{3}}", tail reversed, from=1-4, to=2-5]
	\arrow["{\mathfrak{C}_{m_{1}}}", tail reversed, from=2-1, to=1-2]
	\arrow["{\mathfrak{C}_{4}}", tail reversed, from=2-5, to=3-3]
	\arrow["{\mathfrak{C}_{n}}", tail reversed, from=3-2, to=2-1]
	\arrow["{\mathfrak{C}_{n-1}}", tail reversed, from=3-3, to=3-2]
\end{tikzcd}\]
can only be satisfied if $p_{1}=p_{2}=\cdots=p_{n}$.
\end{enumerate}
\end{defin}

\begin{thm}\label{thm:simsat}
Let $\lambda\geq 1$. Suppose that the set of operators, $\mathbb{T}_{\lambda}=\{T_{1},\dots,T_{M}\}$ is a non-degenerate, uniformly bounded simultaneous system and that there is a single choice of $(f_{1},\dots,f_{M})$ for which there are $N$ distinct points $p_{1},\dots, p_{N}$ so that for each $p_{j}$
\begin{enumerate}
\item There is a lower bound $L$ such that
\begin{equation}\left|\prod_{m=1}^{M}T_{m}f_{m}\Big|_{p_{j}}\right|\geq{}L^{M}\prod_{m=1}^{M}\norm{f_{m}}_{\mathcal{H}_{m}}\label{lbassump}\end{equation}
and;
\item There is a priori control on $N$ of the form $N\leq \mathfrak{R}(\lambda)^{-\frac{\epsilon}{M^{2}}}$ for some $\epsilon>0.$
\end{enumerate}
Then there exists a $C$ independent of $B,L$ and $\lambda$ so that
\begin{equation}N^{1-\epsilon}\leq C \begin{cases}
\left(\frac{B}{L}\right)^{2}& M=1\\
\left(\frac{B}{L}\right)^{\frac{2M}{M-1}}& M\geq{}2.\end{cases}\label{Nbound}\end{equation}

\end{thm}

\begin{remark}
The a priori control on $N$ is usually very easy to establish. For the oscillatory integral based problems we will study in Section \ref{sec:ML} we have decay rates
\[\mathfrak{R}(\lambda)\leq \lambda^{-R}\text{ for any }R\in\N\]
while the growth of $N$ is naturally capped by $N^{C(d)}$ where the constant $C(d)$ depends only on the underlying dimension. The a priori control then follows from picking $R>\frac{C(d)M^{2}}{\epsilon}$. If we had the stronger condition that if $p\CC_{m}q$ failed then
\[T_{m}(p)\left(T_{m}(q)\right)^{\star}=0\]
we would not need the a priori estimate for $N$. 
\end{remark}

The remainder of this section is devoted to the proof of Theorem \ref{thm:simsat}. 

\subsection{Proof of Theorem \ref{thm:simsat}} 
Let $X\in X^{\otimes MN}$ be given by
\[X=(x_{1}^{1},x_{1}^{2},\dots,x_{1}^{M},x_{2}^{1},\dots,x_{2}^{M},\dots x_{N}^{1},\dots,x_{N}^{M}).\] We denote elements of $(\mathcal{H}_{1}\times \cdots\mathcal{H}_{M})^{N}$ by
\[\G=(g_{1}^{1},\dots,g_{M}^{1},g_{1}^{2},\dots,g_{M}^{2},\dots, g_{1}^{N},\dots,g_{M}^{N})\quad g_{m}^{j}\in \mathcal{H}_{m}.\]
Define $\T:(\mathcal{H}_{1}\times \cdots\mathcal{H}_{M})^{N}\to (X^{\otimes MN})^{\C}$ by
\[\T  \G (P) =\prod_{j=1}^{N}\prod_{m=1}^{M}T_{m}g_{m}^{j}(x_{m}^{j}).\]
From our lower bound assumption \eqref{lbassump} can easily see that there is a function $\F$ so that $\T\F$ is large at $(N!)^{M}$ points. Namely let
\[\F=(f_{1},\dots,f_{m},f_{1},\dots,f_{m},\dots,f_{1},\dots,f_{m})\]
and set
\[P_{\ee}=(p_{1}^{1},p_{1}^{2},\dots,p_{1}^{d},p_{2}^{1},\dots,  p^{d}_{N})\]
where each $p_{j}^{m}=p_{j}$ (the superscript index acts as a tagging). 
Notice that $\T\F\Big|_{P}\geq L^{NM}\prod_{m=1}^{M}\norm{f_{m}}_{\mathcal{H}_{m}}$ so long as $P$ can be obtained by permuting the entries of $P_{\ee}$. We introduce notation to exploit that symmetry. Let $S_{N}$ be the set of permutations of $N$ objects and set $D=|S_{N}|^{M}$. 
Let $\bsigma=(\sigma_{1},\sigma_{2},\dots,\sigma_{M})\in (S^{N})^{M}$ and denote
\[P_{\bsigma}=(p_{\sigma_{1}(1)}^{1},p_{\sigma_{2}(1)}^{2},\dots,p_{\sigma_{M}(1)}^{M},p_{\sigma_{1}(2)}^{1},\dots,p_{\sigma_{M}(N)}^{M}).\]
Then
\[\T\F\Big|_{P_{\ee}}=\T\F\Big|_{P_{\bsigma}}\geq L^{MN}\prod_{m=1}^{M}\norm{f_{m}}_{\mathcal{H}}\quad \bsigma\in(S_{N})^{M}.\]
Accordingly we want to study the evaluation of $\T\F$ on the set of points $\P=\{P_{\bsigma}\}_{\bsigma\in (S_{N})^{M}}$. Let $\bdelta_{\P}$ be the measure given by
\begin{equation}
\bdelta_{\P}=\frac{1}{D}\sum_{\bsigma\in (S_{N})^{M}}\bdelta_{P_{\bsigma}}\label{measuredef}\end{equation}
where $\bdelta_{P_{\bsigma}}$ is the usual point measure supported at $P_{\bsigma}$. Then we define $\T_{\P}\F=\T\F\Big|_{\P}$ and equip $\P$ with the $\bdelta_{\P}$ measure to create a Hilbert space, $L^{2}(\bdelta_{\P})$. Then note that if we normalise so that $\norm{\F}=1$,
\[L^{MN}\leq e^{i\theta}\langle \1_{\P},\T_{\P}\F\rangle\]
for some $\theta\in\R$. 
and so
\begin{equation}L^{2MN}\leq \langle \T_{\P}^{\star}\1_{\P},\T_{\P}^{\star}\1_{\P}\rangle =\langle \1_{\P},\T_{\P}\T_{\P}^{\star}\1_{\P}\rangle\label{lowerbnd}\end{equation}
We denote the kernel of $\T_{\P}\T_{\P}^{\star}$ by $\U(P,Q)$,
\begin{equation}\U(P,Q)=\prod_{m=1}^{M}\prod_{j=1}^{N}T_{m}(p_{j}^{m})\left(T_{m}(q_{j}^{m})\right)^{\star}\label{Uldef}\end{equation}

We now record a useful symmetry lemma. 

\begin{lem}\label{lem:inv}
Let 
\begin{equation}
Z=\T_{\P}\T_{\P}^{\star}\1_{\P}\label{Ualphadef}
\end{equation}
then $Z(P_{\bsigma})=Z(P_{\ee})$. That is $Z$ is constant on the support $\bdelta_{\P}$. Therefore $\1_{\P}$ is an eigenfunction of $\T_{\P}\T_{\P}^{\star}$.
\end{lem}
\begin{proof}
Note that 
\begin{align*}
Z(P_{\bsigma})&=\T_{\P}\T_{P}^{\star}\1_{\P}\Big|_{P_{\bsigma}}\\
&=\int \U(P_{\bsigma},Q)d\bdelta_{\P}(Q)\end{align*}
Since
\[\U(P_{\bsigma},Q)=\prod_{m=1}^{M}\prod_{j=1}^{N}T_{m}(p^{m}_{\simga_{m}(j)})(T_{m}(q_{j}^{m}))^{\star}\]
we can reorder the product, in $j$, to write
\begin{align*}
Z(P_{\bsigma})&=\int \left(\prod_{m=1}^{M}\prod_{j=1}^{N}T_{m}(p^{m}_{j})\left(T_{m}(q_{\sigma^{-1}_{m}(j)}^{m})\right)^{\star}\right)d\bdelta_{\P}(Q)\\
&=\int \U(P_{\ee},Q)d\bdelta_{\P}(Q)\\
&=Z(P_{\ee}).\end{align*}
Where to progress from the first to the second line we performed a change of variables $q_{\sigma^{-1}_{m}(j)}\to q_{j}^{m}$ under which the measure $\bdelta_{\P}$ is invariant.

\end{proof}

We know that
\[L^{2MN}\leq  \langle \1_{\P},\T\T^{\star}\1_{\P}\rangle,\]
so we can conclude that if $\Lambda$ is the eigenvalue associated with $\1_{\P}$ then
\begin{equation}\Lambda\geq L^{2MN}.\label{alphalb}\end{equation}

We want to produce loop diagrams involving the relations $\CC_{1},\dots,\CC_{M}$. Since we are working now on finite dimensional vector spaces we reduce the problem to one about matrices. Assume we have some ordering on the elements of $(S_{N})^{d}$. For convenience we refer to the $\btau$ element of a vector similarly we talk about the $(\bsigma,\btau)$ element of the matrix.  Let
\[\mathcal{V}_{\bsigma}(\btau)=\begin{cases}
1& \btau=\bsigma\\
0 & \btau\neq \bsigma.\end{cases}\]
Then $\{V_{\bsigma}\mid\bsigma\in (S_{N})^{M}\}$ gives us a convenient basis in which to express the operator $\T_{\P}\T_{\P}^{\star}$. 

We write
\begin{equation}\T_{\P}\T_{\P}^{\star}= \mathcal{W}_{F}\label{matrixdef}\end{equation}
where $\mathcal{W}$ is the $(S_{N})^{M}\times (S_{N})^{M}$ matrix associated with $\T_{\P}\T_{\P}^{\star}$.  Since it represents an operator in $T_{\P}T_{\P}^{\star}$ format each $\mathcal{W}$ is Hermitian positive semi-definite. In this context the matrix $\mathcal{W}$ becomes our energy matrix, it records all possible interactions between points $p_{j}$, $p_{j}$.   With this notation then
\[\mathcal{W}(\bsigma,\btau)=\frac{1}{D}\U(P_{\bsigma},P_{\btau}).\]

To fix our ideas, let's first consider the $(\bsigma,\btau)$ element of the matrix $\mathcal{W}$, denoted by $\mathcal{W}(\bsigma,\btau)=\frac{1}{D}\U(P_{\bsigma},P_{\btau})$. We will use our conditions on $\U(X,Y)$ to show that if we avoid rapid decay we must produce a path of points connected by the relations $\CC_{1},\dots,\CC_{M}$. We refer to this process as producing a weaving, see Figure \ref{fig:ploop}, through $\U(P_{\bsigma},P_{\btau})$. 
Recall that
\[\U(P,Q)=\prod_{m=1}^{M}\prod_{j=1}^{N}T_{m}(p_{j}^{m})(T_{m}(q_{j}^{m}))^{\star}\]
where for each pair $(m,j)$ one of the following occurs
\begin{enumerate}
\item $p_{j}^{m}\CC_{m}q_{j}^{m}$; or
\item $|T_{m}(p_{j}^{m})(T_{m}(q_{j}^{m}))^{\star}|\leq \mathfrak{R}(\lambda)$.
\end{enumerate}
We write the first option diagrammatically as 
\[p_{j}^{m}\xleftrightarrow[\CC_{m}]{}p_{j'}^{m}\]
this is the only option to avoid decay. 
So consider $\U(P_{\bsigma},P_{\btau})$ for some fixed $(\bsigma,\btau)$.  If we wish to avoid decay we must have
\[p_{\sigma_{m}(j)}^{m}\xleftrightarrow[\CC_{m}] {}p_{\tau_{m}(j)}^{m}\]
for all pairs $(m,j)$. Starting with $m=1$ this requires
\[p_{\sigma_{1}(j)}^{1}\xleftrightarrow[\CC_{1}] {}p_{\tau_{1}(j)}^{1}.\]
Now
\[p^{1}_{\tau_{1}(j)}=p^{2}_{\tau_{1}(j)}=p^{2}_{\tau_{2}\tau_{2}^{-1}\tau_{1}(j)}\]
so we can now add in information about the relation $\CC_{2}$. We refer to the permutation $\tau_{2}^{-1}\tau_{1}$ as the $1\to 2$ transition. To avoid rapid decay from any of the $m=2$ terms in the product defining $\U(P_{\bsigma},P_{\btau})$ we require
\[p^{1}_{\tau_{1}(j)}=p^{2}_{\tau_{2}\tau_{2}^{-1}\tau_{1}(j)}\xleftrightarrow[\CC_{2}]{}p_{\sigma_{2}\tau_{2}^{-1}\tau_{1}(j)}.\]
Putting these together we form a chain
\[p_{\sigma_{1}(j)}^{1}\xleftrightarrow[\CC_{1}] {}p_{\tau_{1}(j)}^{1}\xleftrightarrow[\CC_{2}]{}p_{\sigma_{1}\tau_{2}^{-1}\tau_{1}(j)}^{2}.\]
Since 
\[p_{\sigma_{2}\tau_{2}^{-1}\tau_{1}(j)}^{2}=p^{3}_{\sigma_{3}\sigma_{3}^{-1}\sigma_{2}\tau_{2}^{-1}\tau_{1}(j)}\]
we can repeat the process to add another link
\[p_{\sigma_{1}(j)}^{1}\xleftrightarrow[\CC_{1}] {}p_{\tau_{1}(j)}^{1}\xleftrightarrow[\CC_{2}]{}p_{\sigma_{1}\tau_{2}^{-1}\tau_{1}(j)}^{2}\xleftrightarrow[\CC_{3}]{}p^{3}_{\tau_{3}\sigma_{3}^{-1}\sigma_{2}\tau_{2}^{-1}\tau_{1}(j)}.\]
We continue in this manner until we have used all relations $\CC_{1},\dots,\CC_{M}$. For the resultant path to be a closed loop we need to eventually return to $p_{\sigma_{1}(j)}$. For the moment we will focus on the case $M$ even. For even $M$ the loop condition is
\begin{align}\sigma_{M}\tau_{M}^{-1}\tau_{M-1}\cdots\tau_{2}^{-1}\tau_{1}&=\sigma_{1}\nonumber\\
\sigma_{1}^{-1}\sigma_{M}\tau_{M}^{-1}\tau_{M-1}\cdots\tau_{2}^{-1}\tau_{1}&=e.\label{loopcond}\end{align}
Later we will treat the case when $M$ is odd by embedding into a $M+1$ system.

Figure \ref{fig:ploop} depicts this weaving process (for even $M$) for $j=1$ for a $(\bsigma,\btau)$ that obeys \eqref{loopcond}. 

\begin{figure}[h!]\label{fig:ploop}
\begin{center}
\includegraphics[scale=0.8]{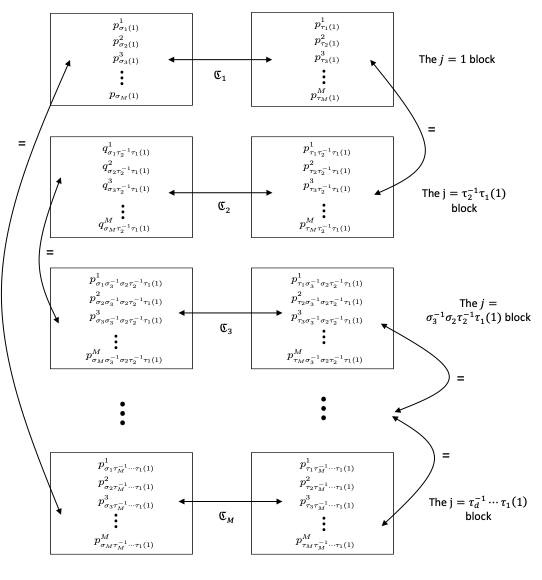}
\end{center}
\caption{If we wish to avoid rapid decay we must connect up all points in a path. If $(\bsigma,\btau)$ obeys \eqref{loopcond} then the path is a loop}

\end{figure}

There are two important facts that come out of the weaving process which we will record here
\begin{enumerate}
\item The condition \eqref{loopcond} depends on the transitions $\sigma_{1}^{-1}\sigma_{M}$, $\tau_{M}^{-1}\tau_{M-1}$, etc rather than on individual $\sigma_{m}$ and $\tau_{m}$. So we should average over all $P_{\bsigma}$, $P_{\btau}$ which have the same transitions. 
\item For every $j=1,\dots,N$ we produce a loop, each loop involved $d$ points. We could start the loop at any other $m$ (rather than $m=1$) the loop condition remains the same. We could also start from $\tau_{1}(j)$ rather than $\sigma_{1}(j)$. Doing so switches the roles of $\bsigma$ and $\btau$ in the loop condition (and so is associated with taking the Hermitian conjugate of the related matrix).  

\end{enumerate}

\subsection*{Case 1: $M$ is even}
We construct matrix $\mathcal{M}$ with whose trace is given by a sum  of loops. By obtaining a lower bound on the trace (from our known large eigenvalue) we can then obtain control on the size of $N$.

First we will define matrices that average over elements with the same transitions. Here and henceforth we will use mod $M$ clock arithmetic when describing the $\sigma_{m}$ so that $\sigma_{M}=\sigma_{0}$.

\begin{defin}
For $(\bsigma,\btau)\in (S_{N})^{M}\times (S_{N})^{M}$ we say that
\[\bsigma\sim_{ev}\btau\]
if for all even $m\leq M$
\begin{equation}\sigma_{m}^{-1}\sigma_{m-1}=\tau^{-1}_{m}\tau_{m-1}. \label{eqrefevdefn}\end{equation}
We say that
\[\bsigma\sim_{odd}\btau\]
if for all odd $m\leq M$
\begin{equation}\sigma_{m}^{-1}\sigma_{m-1}=\tau_{m}^{-1}\tau_{m-1}.\label{eqrefodddef}\end{equation}
\end{defin}

Note that  $\sim_{ev}$ and $\sim_{odd}$ are equivalence relations.  Let $|[\bsigma]_{ev}|$ be the size of the equivalence class of $\bsigma$ under $\sim_{ev}$. If $\btau\sim_{ev}\bsigma$ then the $\tau_{m}$ for $m$ even can be chosen independently with the $\tau_{m}$, $m$ odd determined by \eqref{eqrefevdefn}. Therefore, $|[\bsigma]_{ev}|=(S_{N})^{M/2}=D^{1/2}$. A similar argument shows that the $\sim_{odd}$ equivalence classes have the same size.

We now define matrices associated with the equivalence relations.
\begin{equation}
\mathcal{A}_{ev}(\bsigma,\btau)=\frac{1}{D^{1/2}}
\begin{cases}
1 &\bsigma\sim_{ev}\btau\\
0 &\text{otherwise,}\end{cases}\label{avevdef}
\end{equation}
\begin{equation}
\mathcal{A}_{odd}(\bsigma,\btau)=\frac{1}{D^{1/2}}
\begin{cases}
1 & \bsigma \sim_{odd}\btau\\
0 & \text{otherwise.}\end{cases}\label{avodddef}\end{equation}

Clearly both $\mathcal{A}_{ev}$ and $\mathcal{A}_{odd}$ are an Hermitian projections.  Note that $\mathcal{X}_{\1_{\P}}$ is an eigenvector (with eigenvalue $1$) of both $\mathcal{A}_{ev}$ and $\mathcal{A}_{odd}.$

We will study the matrix
\begin{equation}
\mathcal{M}=\left(\mathcal{A}_{ev}\mathcal{A}_{odd}+\mathcal{A}_{odd}\mathcal{A}_{ev}\right)\mathcal{W}\label{Mdef}\end{equation}
and will see that its trace consists of a sum over loops. Both $\mathcal{W}$ and $\left(\mathcal{A}_{ev}\mathcal{A}_{odd}+\mathcal{A}_{odd}\mathcal{A}_{ev}\right)$ are clearly Hermitian. The matrix $\mathcal{W}$ is positive semi-definite (as it is $TT^{\star}$ of another operator) and, as we will now show, so is $\left(\mathcal{A}_{ev}\mathcal{A}_{odd}+\mathcal{A}_{odd}\mathcal{A}_{ev}\right)$. Therefore their product, $\mathcal{M}$ has only positive eigenvalues and we can lower bound the trace by the single eigenvalue that we know.

For convenience denote
\begin{equation}\mathcal{A}=\left(\mathcal{A}_{ev}\mathcal{A}_{odd}+\mathcal{A}_{odd}\mathcal{A}_{ev}\right).\label{Amatrixdef}\end{equation}
 To analyse the eigenvalues of $\mathcal{A}$ we need to introduce an auxiliary operator. 
 
 Let
 \begin{equation}
 \mathcal{C}(\bsigma,\btau)=\begin{cases}
 1 & \sigma_{m+1}=\tau_{m}\text{ for all }1\leq m\leq M\\
 0 & \text{otherwise.}\end{cases}\label{cycledef}\end{equation}
 Note that $\mathcal{C}$ permutes the entries of a vector and so is unitary. In fact $\mathcal{C}^{M}=\Id$ so the eigenvalues of $\mathcal{C}$ are the $M$-th roots of unity. Associated with the matrix $\mathcal{C}$ we define the map $\mathfrak{c}:(S_{N})^{M}\to(S_{N})^{M}$ by
\[\mathfrak{c}(\sigma_{1},\sigma_{2},\dots,\sigma_{M})=(\sigma_{2},\sigma_{3},\dots,\sigma_{M},\sigma_{1}).\]
 Note that
\[\mathcal{C}\mathcal{V}_{\bsigma}=\mathcal{V}_{\mathfrak{c}^{-1}(\sigma)}\]
 
 \begin{lem}\label{lem:similarities}
 The matrices $\mathcal{A}_{ev}$ and $\mathcal{A}_{odd}$ are similar with similarity transform $\mathcal{C}$. The matrix $\mathcal{C}^{2}$ commutes with both $\mathcal{A}_{ev}$ and $\mathcal{A}_{odd}$. 
 \end{lem}
 
 \begin{proof}
 The matrix
\[\mathcal{C}\mathcal{A}_{ev}\mathcal{C}^{-1}(\bsigma,\btau)=\frac{1}{D^{1/2}}\begin{cases}
 1 & \mathfrak{c}(\bsigma)\sim_{ev}\mathfrak{c}(\btau)\\
 0 & \text{otherwise}\end{cases}\]
 The requirement $\mathfrak{c}(\bsigma)\sim_{ev}\mathfrak{c}(\btau)$ enforces
\begin{align*}
\sigma_{3}^{-1}\sigma_{2}&=\tau_{3}^{-1}\tau_{2}\\
\sigma_{5}^{-1}\sigma_{4}&=\tau_{5}^{-1}\tau_{4}\\
\vdots &\\
\sigma_{M-1}^{-1}\sigma_{M-2}&=\tau_{M-1}^{-1}\tau_{M-2}\\
\sigma_{1}^{-1}\sigma_{M}&=\tau_{1}^{-1}\tau_{M}\end{align*}
or $\bsigma\sim_{odd}\btau$. So indeed $\mathcal{A}_{odd}=\mathcal{C}\mathcal{A}_{ev}\mathcal{C}^{-1}$. Now
\[\mathcal{C}^{2}\mathcal{A}_{ev}(\bsigma,\btau)=\frac{1}{D^{1/2}}\begin{cases}
1 & \mathfrak{c}^{2}(\bsigma)\sim_{ev}\btau\\
0 & \text{otherwise}\end{cases}\]
 The requirement $\mathfrak{c}^{2}(\bsigma)\sim_{ev}\btau$ enforces
\begin{align*}
\sigma_{4}^{-1}\sigma_{3}&=\tau_{2}^{-1}\tau_{1}\\
\sigma_{6}^{-1}\sigma_{5}&=\tau_{4}^{-1}\tau_{3}\\
\vdots &\\
\sigma_{M}^{-1}\sigma_{M-1}&=\tau_{M-2}^{-1}\tau_{M-3}\\
\sigma_{2}^{-1}\sigma_{1}&=\tau_{M}^{-1}\tau_{M-1}\end{align*}
which is the same as requiring $\bsigma\sim_{ev} \mathfrak{c}^{-2}(\btau)$. A similar argument shows that $\mathcal{C}^{2}$ commutes with $\mathcal{A}_{odd}$. The key point driving the commutation identities is that the set of even (or odd) transitions are mapped to even (or odd) transitions by $\mathcal{C}^{2}$.
 \end{proof}
 
 \begin{lem}\label{lem:posdef}
 The matrix $\mathcal{A}$ given by \eqref{Amatrixdef} is positive semi-definite. 
 \end{lem}
 
 \begin{proof}
 Clearly it is Hermitian. Since both $\mathcal{A}_{ev}$ and $\mathcal{A}_{odd}$ are projections their eigenvalues are in $\{0,1\}$ we denote $\mathfrak{I}_{ev}$ and $\mathfrak{I}_{odd}$ the subspaces spanned by the  eigenvectors of eigenvalue 1 associated with $\mathcal{A}_{ev}$ and $\mathcal{A}_{odd}$ respectively.
 
Because $\mathcal{A}_{ev}$ and $\mathcal{A}_{odd}$ are projections we can write
\[(\mathcal{A}_{ev}+\mathcal{A}_{odd})^{2}=\mathcal{A}_{ev}+\mathcal{A}_{odd}+\mathcal{A}.\]
 So if we were able to show that the non-zero eigenvalues of $\mathcal{A}_{ev}+\mathcal{A}_{odd}$ were all at least $1$ then we could conclude that the non-zero eigenvalues of $\mathcal{A}$ are positive. This is what we shall do. 
 
 First, using the similarity relationship,
\[\mathcal{A}_{ev}+\mathcal{A}_{odd}=\mathcal{A}_{ev}+\mathcal{C}\mathcal{A}_{ev}\mathcal{C}^{-1}\]
 and so
\[(\mathcal{A}_{ev}+\mathcal{A}_{odd})\mathcal{C}=\mathcal{A}_{ev}\mathcal{C}+\mathcal{C}\mathcal{A}_{ev}.\]
 However also using the commutation relationship for $\mathcal{C}^{2}$
 \begin{align*}
 \mathcal{C}(\mathcal{A}_{ev}+\mathcal{A}_{odd})&=\mathcal{C}\mathcal{A}_{ev}+\mathcal{C}^{2}\mathcal{A}_{ev}\mathcal{C}^{-1}\\
 &=\mathcal{A}_{ev}\mathcal{C}+\mathcal{C}\mathcal{A}_{ev}.\end{align*}
Therefore $\mathcal{C}$ commutes with $\mathcal{A}_{ev}+\mathcal{A}_{odd}$. So we can find a complete set of joint eigenvectors for $\mathcal{C}$ and $\mathcal{A}_{ev}+\mathcal{A}_{odd}$. Now assume that $v$ is just such a joint eigenvector with $\kappa\neq 0$ eigenvalue for $\mathcal{A}_{ev}+\mathcal{A}_{odd}$. Its eigenvalue for $\mathcal{C}$ is a $M$-th root of unitary which we write as $e^{i\theta}$. We will assume that $v$ is normalised so $\norm{v}=1$. We compute
 \begin{align*}
 \kappa v&=(\mathcal{A}_{ev}+\mathcal{C}\mathcal{A}_{ev}\mathcal{C}^{-1})v\\
 & =\mathcal{A}_{ev}v+e^{-i\theta}\mathcal{C}\mathcal{A}_{ev}v\end{align*}
 Now $\mathcal{A}_{ev}v\in \mathfrak{I}_{ev}$ therefore we can write $v$ as
\[v=w_{1}+w_{2}\]
 where $w_{1}\in \mathfrak{I}_{ev}$ and $w_{2}=e^{-i\theta}\mathcal{C}w_{1}$. Since $\mathcal{A}_{ev}$ and $\mathcal{A}_{odd}$ are similar with similarity transform $\mathcal{C}$ we have that  $w_{2}\in\mathfrak{I}_{odd}$. Since $\mathcal{C}$ is unitary $\norm{w_{1}}=\norm{w_{2}}$. Now
 \begin{align*}\kappa&=\langle v,(\mathcal{A}_{ev}+\mathcal{A}_{odd})v\rangle\\
 &=\langle v,w_{1}+\mathcal{A}_{odd}w_{1}+\mathcal{A}_{ev}w_{2}+w_{2}\rangle\\
 &=1+\langle w_{1},\mathcal{A}_{odd}w_{1}\rangle+\langle w_{2},\mathcal{A}_{odd}w_{1}\rangle+\langle w_{1},\mathcal{A}_{ev}w_{2}\rangle+\langle w_{2},\mathcal{A}_{ev}w_{2}\rangle\\
 &=1+\norm{w_{1}}^{2}\left(\langle \hat{w}_{1},\hat{w}_{2}\rangle+\langle \hat{w}_{2},\hat{w}_{1}\rangle+\langle \hat{w}_{1},\mathcal{A}_{odd}\hat{w}_{1}\rangle+\langle \hat{w}_{2},\mathcal{A}_{ev}\hat{w}_{2}\rangle\right).\end{align*}
 The matrices $\mathcal{A}_{ev}$ and $\mathcal{A}_{odd}$ are orthogonal projections so both $\langle \hat{w}_{1},\mathcal{A}_{odd}\hat{w}_{1}\rangle$ and $\langle \hat{w}_{2},\mathcal{A}_{ev}\hat{w}_{2}\rangle$ are positive real numbers. They  are equal to the magnitudes of the orthogonal projection of $\hat{w}_{1}$ onto $\mathfrak{I}_{odd}$ and of $\hat{w}_{2}$ onto $\mathfrak{I}_{ev}$ respectively. On the other hand $|\langle \hat{w}_{1},\hat{w}_{2}\rangle|$ is the magnitude of the projection of $\hat{w}_{1}$ onto $\hat{w}_{2}$ or vice-versa. Therefore since $\hat{w}_{1}\in\mathfrak{I}_{ev}$ and $\hat{w}_{2}\in\mathfrak{I}_{odd}$,
\[ |\langle \hat{w}_{1},\hat{w}_{2}\rangle|\leq\langle \hat{w}_{2},\mathcal{A}_{ev}\hat{w}_{2}\rangle\quad\text{and}\quad|\langle \hat{w}_{1},\hat{w}_{2}\rangle|\leq \langle \hat{w}_{1},\mathcal{A}_{odd}\hat{w}_{1}\rangle\]
 and so
 \[\kappa\geq{}1.\]
\end{proof}

So we may conclude that $\mathcal{A}$ is positive semi-definite and therefore that $\mathcal{M}$ has positive eigenvalues. We now use that to state a lower bound on the trace of $\mathcal{M}$. 

\begin{prop}\label{prop:lb}
Let $\mathcal{M}=\mathcal{A}\mathcal{W}$. Then
\[ \Tr(\mathcal{M})\geq{}L^{2MN}.\]
\end{prop} 

\begin{proof}

We will exploit the fact that the vector of $1$s is a joint eigenvector (with eigenvalues $1$ and $\Lambda$ respectively) of $\mathcal{A}_{ev}$, $\mathcal{A}_{odd}$ and $\mathcal{W}$. Therefore it is an eigenvector of $\mathcal{M}$ with eigenvalue $\Lambda$. From \eqref{alphalb} we have seen that
\[\Lambda\geq{}L^{2MN}.\]
Since all the eigenvalues of $\mathcal{M}$ are non-negative
\[\Tr(\mathcal{M})\geq{}L^{2MN}.\]

\end{proof}

On the other hand we will show, in Proposition \ref{prop:ub}, that the trace of $\mathcal{M}$ involves sums over loops and this will lead to an upper bound on the trace. The two bounds together control the growth of $N$. First we prove that $\mathcal{A}$ does indeed give us the sum over loops. We say that $\bsigma$ weaves with $\btau$, notated $\bsigma\weave\btau$ if
\begin{equation}
\sigma_{1}^{-1}\sigma_{M}\tau_{M}^{-1}\tau_{M-1}\cdots\tau^{-1}_{2}\tau_{1}=e\label{weavingdef}\end{equation}

\begin{lem}\label{lem:Aweave}
The matrices $\mathcal{A}_{odd}\mathcal{A}_{ev}$ and $\mathcal{A}_{ev}\mathcal{A}_{odd}$ are given by,
\begin{equation}
\mathcal{A}_{odd}\mathcal{A}_{ev}=\frac{1}{D^{1-\frac{1}{M}}}\begin{cases}
1 & \bsigma \weave \btau\\
0 &\text{ otherwise.}
\end{cases}\label{weaving1}\end{equation}
\begin{equation}
\mathcal{A}_{ev}\mathcal{A}_{odd}=\frac{1}{D^{1-\frac{1}{M}}}\begin{cases}
1 & \btau \weave \bsigma\\
0 &\text{ otherwise.}
\end{cases}\label{weaving2}\end{equation}
\end{lem}

\begin{proof}
Since they are Hermitian conjugates of each other it is only necessary to prove \eqref{weaving1}. 
\[\mathcal{A}_{odd}\mathcal{A}_{ev}(\bsigma,\btau)=\frac{1}{D}\sum_{\substack{\bmu\sim_{odd}\bsigma\\ \bmu \sim_{ev} \btau}}1.\]
Let's first examine a necessary condition for $\mathcal{A}_{odd}\mathcal{A}_{ev}(\bsigma,\btau)$ to be nonzero. Clearly we require that there be at least one $\bmu$ with the same odd transitions as $\bsigma$ and the same even transitions as $\btau$. Now
\[(\mu_{1}^{-1}\mu_{M})(\mu_{M}^{-1}\mu_{M-1})\cdots (\mu_{3}^{-1}\mu_{2})(\mu_{2}^{-1}\mu_{1})=e\]
so we must have that
\[(\sigma_{1}^{-1}\sigma_{M})(\tau_{M}^{-1}\tau_{M-1})\cdots(\sigma_{3}^{-1}\sigma_{2})(\tau_{2}^{-1}\tau_{1})=e\]
or $\bsigma\weave \btau$. If $\bmu$ has the same odd transitions as $\bsigma$ and the same even transitions as $\btau$ and we know one $\mu_{m}$ (for instance $\mu_{1}$) the other $\mu_{m'}$ can be calculated from the transitions. So if $(\bsigma,\btau)$ satisfy the necessary condition $\bsigma\weave\btau$ there are $|S_{N}|=D^{\frac{1}{M}}$ $\bmu$ that can appear in the sum. Therefore we have established \eqref{weaving1}. 
\end{proof}

We are now in a position to find an upper bound on the trace of $\mathcal{M}$.

\begin{prop}\label{prop:ub}
If $\mathcal{M}=\mathcal{A}\mathcal{W}$ then if $\epsilon>0$ arises from the a priori control on $N$ (i.e. $N\leq\mathfrak{R}(\lambda)^{-\frac{\epsilon}{M^{2}}}$)  then \begin{equation}\Tr(\mathcal{M})\leq 2^{NM}N^{-N(M-1)+\epsilon N}B^{2MN}.
\label{trub}\end{equation}
\end{prop}

\begin{proof}

We have seen that in order to avoid decay of $\U(P_{\bsigma},P_{\btau})$ we can must be able to construct sequences
 \[p_{\sigma_{1}(j)}^{1}\xleftrightarrow[\CC_{1}] {}p_{\tau_{1}(j)}^{1}\xleftrightarrow[\CC_{2}]{}p_{\sigma_{1}\tau_{2}^{-1}\tau_{1}(j)}^{2}\xleftrightarrow[\CC_{3}]{}p^{3}_{\tau_{3}\sigma_{3}^{-1}\sigma_{2}\tau_{2}^{-1}\tau_{1}(j)}\xleftrightarrow[\CC_{4}]{}\cdots\xleftrightarrow[\CC_{d}]{}p^{M}_{\sigma_{M}\tau_{M}^{-1}\tau_{M-1}\cdots\tau_{2}^{-1}\tau_{1}.}\]
and this closes to be a loop if
\begin{equation}\sigma_{M}\tau_{M}^{-1}\tau_{M-1}\cdots\tau_{2}^{-1}\tau_{1}=\sigma_{1}\Rightarrow \sigma_{1}^{-1}\sigma_{M}\tau_{M}^{-1}\tau_{M-1}\cdots\tau_{2}^{-1}\tau_{1}=e\label{loopcond1}\end{equation}
that is $\bsigma\weave\btau$. Note that we could have started from $\tau_{1}$ in that case the condition to close the loop becomes
\begin{equation}\tau_{1}^{-1}\tau_{M}\sigma_{M}^{-1}\sigma_{M-1}\cdots\sigma_{2}^{-1}\sigma_{1}=e\label{looprevcond}\end{equation}
that is $\btau\weave\bsigma$. 

Now consider a trace element of $\mathcal{M}$, $\mathcal{M}(P_{\bsigma},P_{\bsigma})$. This is given by
\[\frac{1}{D^{2-\frac{1}{M}}}\left(\sum_{\bsigma\weave\bmu}\U(P_{\bmu},P_{\bsigma})+\sum_{\bmu\weave\bsigma}\U(P_{\bmu},P_{\bsigma})\right)\]
so these elements do obey a loop condition. For notational convenience we will notate the points that make up the loop as $a_{1},\dots,a_{M}$. 

\begin{figure}[h!]\label{fig:bigloop}
\[\begin{tikzcd}
	&& {a_{2}} \\
	{a_{1}} \\
	&&&& {} & {} \\
	& {a_{M}} && \cdots && {a_{3}}
	\arrow["{\CC_{1}}", tail reversed, from=2-1, to=1-3]
	\arrow["{\CC_{2}}", tail reversed, from=1-3, to=4-6]
	\arrow["{\CC_{M-1}}", tail reversed, from=4-4, to=4-2]
	\arrow["{\CC_{3}}", tail reversed, from=4-6, to=4-4]
	\arrow["{\CC_{M}}", tail reversed, from=4-2, to=2-1]
\end{tikzcd}\]
\caption{We produce a closed loop involving $a_{1},\dots,a_{M}$ and the relations $\CC_{1},\dots\CC_{M}$. }
\end{figure}

The non-degeneracy condition on the operators $T_{m}$ tell us that this loop can only occur if $a_{1}=a_{2}=\cdots=a_{M}$. So either the loop is the trivial one or one of the $a_{i}\CC_{i}a_{i+1}$ does not hold. If the $i$-th relation doesn't hold then that term experiences rapid decay. This observation allows us to see that the main contribution to the trace comes from the terms where $\bmu=\bsigma$.  In fact if the permutations are not quite equal but differ only at a small number of sites the decay from the product kernels isn't enough to ignore these terms. To obtain enough decay we need to assume that $\bmu$ and $\bsigma$ disagree at a positive proportion of sites. 

First we will treat the sum where $\bmu\weave\bsigma$  (the other sum follows by the same argument).

Let
\begin{equation}
\mathfrak{D}^{\epsilon}(\bsigma)=\left\{\bmu\,\Bigg| \substack{\bmu\weave\bsigma\\ 
d_{H}(\bmu,\bsigma)\leq \epsilon N}\right\}\label
{diagdef}\end{equation}
where $d_{H}(\cdot,\cdot)$ is the Hamming distance. Now suppose that $\bmu\in (\mathfrak{D}^{\epsilon}(\bsigma))^{c}$. There are $A\geq{}\epsilon N$ pairs $(m,j)$ so that $\mu_{m}(j)\neq\sigma_{m}(j)$, we may as well assume that $A$ is divisible by $M$.  Let $(m,j)$ be a such a pair. We start a weaving process from this point to construct a non-trivial loop
\[p_{\mu_{m}(j)}\xleftrightarrow[\CC_{m}]{}p_{\sigma_{m}(j)}\xleftrightarrow[\CC_{m+1}]{} p_{\mu_{m+1}\sigma_{m+1}^{-1}\sigma_{m}(j)}\xleftrightarrow[\CC_{m+3}]{}\cdots\xleftrightarrow[\CC_{m-1}]{}p_{\mu_{m}(j)}.\]
Therefore one of the kernels
\begin{align*}&T_{m}(p_{\mu_{m}(j)})(T_{m}(p_{\sigma_{m}(j)}))^{\star}\\
&T_{m+1}(p_{\mu_{m+1}\sigma_{m+1}^{-1}\sigma_{m}(j)})(T_{m+1}(p_{\sigma_{m}(j)}))^{\star}\\
&\vdots\\
&T_{m-1}(p_{\mu_{m}(j)})(T(p_{\sigma_{m-1}^{-1}\cdots \sigma_{m}(j)}))^{\star}\end{align*}
must rapidly decay as $\mathfrak{R}(\lambda)$. Possibly we have used some other $(m',j')$ such that $\mu_{m'}(j')\neq\sigma_{m'}(j')$ in this loop but at most we have only used $M$ of them, so we have $A-M$ remaining. Pick one of those and repeat the process to general another decaying kernel. We then have $A-2M$ remaining. Repeat this process $A/M$ times to produce $A/M$ kernels each decaying. Therefore
\begin{align}|\U(P_{\bmu},P_{\bsigma})|&\leq  B^{2MN}\mathfrak{R}(\lambda)^{\frac{A}{M}}\nonumber\\
&\leq B^{2MN}\mathfrak{R}(\lambda)^{\frac{\epsilon N}{M}}.\label{trivterms}\end{align}
So these terms each make a negligible contribution to the trace.

Now it remains to treat the terms where $\bmu\in\mathfrak{D}^{\epsilon}(\bsigma)$. For these terms we accept a trivial  bound of $B^{2}$ (from the kernel bound, assumption \eqref{trivest}). However there are relatively few such terms. Given a particular set of (no more than $\epsilon N$) sites $(m,j)$ where $\mu_{m}(j)\neq \sigma_{m}(j)$ there can be no more than $N^{\epsilon N}$ ways to chose $\bmu$.  There are no more than $2^{MN}$ ways to pick the sets of sites so
\begin{equation}|\mathfrak{D}(\bsigma)|\leq 2^{MN}N^{\epsilon N}.\label{diagonalbnd}\end{equation}

Combining \eqref{trivterms} and \eqref{diagonalbnd} we obtain the estimate
\[|\mathcal{M}(\bsigma,\bsigma)|\leq D^{-2+\frac{1}{M}}2^{NM}B^{2MN}\left(N^{\epsilon N}+D^{1-\frac{1}{M}}\mathfrak{R}(\lambda)^{\frac{\epsilon N}{M}}\right).\]
 Recalling that $D=|S_{N}|^{M}=(N!)^{M}$ we then have that
  \[\Tr(\mathcal{M})\leq  C^{MN}N^{-N(M-1)}B^{2MN}\left(N^{\epsilon N}+N^{(M-1)N}\mathfrak{R}(\lambda)^{\frac{\epsilon N}{M}}\right).\]
  The a priori control on $N$ then allows us to say that
  \[\Tr(\mathcal{M})\leq C^{MN}B^{2MN}N^{-N(M-1)+\epsilon N}.\]
 \end{proof}

We can now finally prove Theorem \ref{thm:simsat}. Putting together the lower and upper bounds for $\Tr(\mathcal{M})$ (Propositions \ref{prop:lb} and \ref{prop:ub} respectively) we have 
\[L^{2MN}\leq C^{MN}N^{-N(M-1)+\epsilon}B^{2MN}\]
 so
\begin{equation}N^{1-\frac{\epsilon}{M-1}}\leq C^{M} \left(\frac{B}{L}\right)^{\frac{2M}{M-1}}.\label{evenNbnd}\end{equation}

 \subsection*{Case 2: $M$ is odd}
 
 We now turn our attention to the case where $M$ is odd. The case where $M=1$ is a little different and we will treat it separately. For $M$ odd and greater than $1$ we embed it in a $M+1$ system. This amounts to adding another condition $\CC_{M+1}$ given by $p\CC_{M+1}q$ if $p=q$.  
 
 Let $\nu$ be the measure on $X$ given by
 \[\nu=\frac{1}{N}\sum_{j=1}^{N}\delta_{p_{j}}\]
 where $\delta_{p_{j}}$ is the standard point mass measure at $p_{j}$. Then define $T_{M+1}:L^{2}(\nu)\to X^{\C}$ by
 \[T_{M+1}f|_{p}=\frac{B}{N^{1/2}}\begin{cases}
0 & p\notin\{p_{1},\dots,p_{N}\}\\
f(p_{j}) & p=p_{j}, \text{ for some }j\in\{1,\dots,N\}.\end{cases}\]
We now consider the system $\mathbb{T}_{\lambda}=\{T_{1},\dots, T_{M},T_{M+1}\}$. 
For any pair of points $(p,q)$,
\[T_{M+1}(p)\left(T_{M+1}(q)\right)^{\star}=\begin{cases}
B^{2} & p=q\\
0 & p\neq q \end{cases}\]
so $\CC_{M+1}$ is the identity relation. Clearly therefore this system still obeys the non-degeneracy condition and by picking $(f_{1},\dots,f_{M},1)$ we can still obtain a lower bound, namely $\left(L^{\frac{M}{M+1}}B^{\frac{1}{M+1}}N^{-\frac{1}{2(M+1)}}\right)^{M+1}$. Therefore we may apply the equation \eqref{evenNbnd} for a   $M+1$ multilinearity,
\begin{align*}
N^{1-\frac{\epsilon}{M-1}}&\leq C^{M+1} \left(\frac{B^{1+\frac{1}{M+1}}N^{\frac{1}{2(M+1)}}}{L^{\frac{M}{M+1}}}\right)^{\frac{2(M+1)}{M}}\\
N^{\frac{M-1}{M}-\frac{\epsilon}{M-1}}&\leq C^{M+1} \frac{B^{2}}{L^{2}}\\
N^{1-\frac{\epsilon M}{(M-1)^{2}}}&\leq C^{M\frac{M+1}{M-1}}\left(\frac{B}{L}\right)^\frac{2M}{M-1}\end{align*}

Finally if $M=1$ the decay and non-degeneracy conditions of Definition \ref{def:system} ensure that
\[T_{1}(p)(T_{1}(q))^{\star}\]
rapidly decays if $p\neq q$. Therefore instead of defining $\mathcal{M}=\mathcal{A}\mathcal{W}$ we define $\mathcal{M}=\mathcal{S}\mathcal{W}$ where
\[\mathcal{S}(\bsigma,\btau)=\frac{1}{N!}.\]
Clearly $\mathcal{S}$ is Hermitian positive semi-definite and the vector of ones is a eigenvector of eigenvalue $1$. Therefore we obtain
\[L^{N}\leq \Tr(\mathcal{M})\]
as before. In this case $\mathcal{W}(\bsigma,\btau)$ decays if $\bsigma\neq \btau$ so from the same analysis as the even $M$ case
\[\Tr(\mathcal(M))\leq C^{N} B^{2N}N^{-N+\epsilon N}\]
yielding
\[N^{1-\epsilon}\leq C\frac{B^{2}}{L^{2}}\]

This completes the proof of Theorem \ref{thm:simsat}.

\section{Multilinear extension estimates}\label{sec:ML}

In this section we apply the method of simultaneous saturation to establish $k$-linear restriction estimates under a mixture of transversality and curvature hypotheses.  Multilinear restriction/extension theory typically replaces the curvature conditions that appear in the classical linear theory with transversality assumptions.  A detailed account of the development of this framework and its motivating heuristics is given in the survey of Bennett \cite{Ben14}.

The $k$-linear extension problem asks for which exponents $(p,q)$ the inequality
\begin{equation}
\norm{\prod_{m=1}^{k}\mathcal{E}_{m}f_{m}}_{L^{p/k}}\le C(p)\prod_{m=1}^{k}\norm{f_{m}}_{L^{q}(H_{m})}\label{eqn:klin}
\end{equation}
holds. Here each $H_{m}$ is a smooth, compactly supported hypersurface, and
\[\mathcal{E}_{m}f(x)
=\int_{H_{m}}e^{i\langle x,\xi\rangle}f(\xi),d\mu_{m}(\xi)\]
is the associated extension operator.  Typically, the hypersurfaces $$(H_{1},\dots,H_{k})$$ are assumed to satisfy a (uniform) transversality condition
\[|\nu_{1}(\xi^{1})\wedge\cdots\wedge\nu_{k}(\xi^{k})|\ge c>0,\]
where $\nu_{m}(\xi^{m})$ denotes the unit normal to $H_{m}$ at the point $\xi^{m}\in H_{m}$.  This transversality assumption replaces curvature in the multilinear setting, however the transversality assumption is motivated by the classical curvature assumption. Consider the linear problem with a curved hypersurface. One may obtain a bilinear formulation satisfying transversality by selecting two well-separated subsets $H_{1},H_{2}\subset H$; the curvature of $H$ ensures the required transversality between the corresponding normals.

It is standard to reduce such extension problems to parameter-dependent versions by localising to a ball of radius $\lambda$ and requiring a uniform in $\lambda$ estimate
\begin{equation}\label{eqn:locklin}
\norm{\prod_{m=1}^{k}\mathcal{E}_{m}f_{m}}_{L^{p/k}(\mathbb{B}_{\lambda})}
\le C(p)\prod_{m=1}^{k}\norm{f_{m}}_{L^{q}(H_{m})}.
\end{equation}
After rescaling, this is equivalent to showing that
\begin{equation}\label{klinconj}
\norm{\prod_{m=1}^{k}\mathcal{E}_{m,\lambda}f_{m}}_{L^{p/k}(\mathbb{B}_{1})}
\le C(p)\lambda^{-dk/p}\prod_{m=1}^{k}\norm{f_{m}}_{L^{q}(H_{m})},
\end{equation}
A weaker version of \eqref{klinconj}, incorporating a $\lambda^{\epsilon}$-loss, asserts that for every $\epsilon>0$ there exists a constant $C_{\epsilon}$ such that
\begin{equation}\label{nearop}
\norm{\prod_{m=1}^{k}\mathcal{E}_{m,\lambda}f_{m}}_{L^{p/k}(\mathbb{B}_{1})}
\le C_{\epsilon}(p)\lambda^{-dk/p+\epsilon}
\prod_{m=1}^{k}\norm{f_{m}}_{L^{q}(H_{m})}.
\end{equation}

The bilinear case $k=2$ has been extensively studied; see, for instance,
\cite{Bour95}, \cite{Wolff01}, \cite{Tao01}, \cite{Tao03}, \cite{TaoVar00},
\cite{LeeVar10}, \cite{Stov17}, \cite{BusSteDetVar17},
\cite{BakLeeLee17}, \cite{TaoVarVeg98}, \cite{Bejen17}, \cite{Var05},
and \cite{Lee11}, \cite{Tem13}. In this paper we are concerned with larger $k$ but restrict our attention to the case $q=2$.

When $q=2$ and $k=1$, inequality \eqref{eqn:klin} recovers the classical Tomas–Stein restriction theorem. In this sense, the $(2,p)$  $k$-multilinear estimates may be viewed as a natural generalisation of the Tomas–Stein result into the multilinear setting.  The case $k=d$ is now fully understood.  Bennett, Carbery, and Tao \cite{BCT} established the sharp, near-optimal $(p,2)$ bounds for $p\ge \frac{2d}{d-1}$ by means of the Loomis-Whitney inequality \cite{LooWhit49} and $d$-linear transversal Kakeya inequality.  Guth subsequently strengthened the multilinear transversal Kakeya estimate \cite{G10}, removing the $\lambda^{\epsilon}$ loss.  More recently, Tao \cite{Tao20} obtained the lossless version of the $d$-linear restriction estimates for $p>\frac{2d}{d-1}$.

When only transversality is assumed, the near-optimal bounds are known to hold for $(p,2)$ with $p\ge \frac{2k}{k-1}$, see \cite{BCT} and \cite{BenCarAntWri05}. If only transversality assumptions hold these estimates are the best available. We will now see how we can use Theorem \ref{thm:simsat} to recover this bound for $k=d$ (the cases $k<d$ follow in the same fashion).

The system of operators is obviously given by $\mathbb{T}_{\lambda}=\{\mathcal{E}_{1,\lambda},\dots,\mathcal{E}_{d,\lambda}\}$. Now we need to establish the decay and non-degeneracy conditions. Since
\[|\nu_{1}(\xi^{1})\wedge\cdots\wedge \nu_{k}(\xi^{k})|\geq{}c\]
we can parameterise the hypersurfaces
\begin{align*}
&H_{1}=\{\xi\mid \xi_{1}=\Sigma_{1}(\xi_{2},\dots,\xi_{d})\}\\
&H_{2}=\{\xi \mid \xi_{2}=\Sigma_{2}(\xi_{1},\xi_{3},\dots,\xi_{d})\}\\
& \vdots\\
& H_{d}=\{\xi \mid \xi_{d}=\Sigma_{d}(\xi_{1},\dots,\xi_{d-1})\}.\end{align*}

So 
\[\mathcal{E}_{m,\lambda}\mathcal{E}_{m,\lambda}^{\star}=\int  e^{i\lambda \left((x_{m}-y_{m})\Sigma_{m}(\bar{\xi})+\langle \bar{x}-\bar{y},\bar{\xi}\rangle\right)}b(\bar{\xi})d\bar{\xi}\]
where $b$ is a smooth compactly supported (on a neighbourhood of the origin) function and $\bar{\xi}=(\xi_{1},\dots,\xi_{m-1},\xi_{m+1},\dots,\xi_{d})$. Clearly
\[|\mathcal{E}_{m}(p)(\mathcal{E}_{m}(q))^{\star}|\leq B=1.\]
To obtain decay conditions consider the critical points of the phase, to achieve a critical point we require that
\[(x_{m}-y_{m})\nabla_{\bar{\xi}}\Sigma_{m}+\bar{x}-\bar{y}=0\]
that is for each $i=1,\dots,d$, $i\neq m$
\[\left[\begin{array}{c}
0 \\
\vdots \\
\partial_{\xi_{i}}\Sigma_{m}\\
0\\
\vdots\\
1\\
0\\
\vdots\\
0\end{array}\right]\cdot (x-y)=0.\]
So $(x-y)$ must point in the normal direction to $H_{m}$ (or be zero). Let $\nu_{m}$ be the normal at $\bar{\xi}=0$, we can assume that the surfaces are small so that $|\nu_{m}(\bar{\xi})-\nu_{m}|\leq \tilde{\epsilon}$ for small $\tilde{\epsilon}$. So if $|x-y|\geq{}\lambda^{-1+\epsilon}$ but 
\[\left|\frac{x-y}{|x-y|}-(\pm\nu_{m})\right|\geq{}2\tilde{\epsilon}\]
there can never be a critical point and 
\[\left|\int  e^{i\lambda \left((x_{m}-y_{m})\Sigma_{m}(\bar{\xi})+\langle \bar{x}-\bar{y},\bar{\xi}\rangle\right)}b(\bar{\xi})d\bar{\xi}\right|=O(\lambda^{-\infty}).\]
Therefore for this system we have $\mathfrak{D}(\lambda)=\lambda^{-R}$ for any $R\in \N$ and $\mathfrak{C}_{m}$ is given by
\[p\mathfrak{C}_{m}q=\left\{(p,q)\in \R^{d}\times \R^{d}\middle|\substack{ |p-q|\leq \lambda^{-1+\epsilon}\text{ or }\\ \left|\frac{p-q}{|p-q|}-(\pm\nu_{m})\right|\leq 2\tilde{\epsilon}}\right\}.\]
Now let's see the non-degeneracy condition. Since
\[|\nu_{1}\wedge\cdots\wedge\nu_{d}|\geq{}c\]
a suitable choice of $\tilde{\epsilon}$ will ensure that if the diagram
\[\begin{tikzcd}
	& {p_{2}} && { p_{3}} \\
	{ p_{1}} &&&& { p_{4}} \\
	& {p_{n}} & \cdots & {}
	\arrow["{\mathfrak{C}_{m_{2}}}", tail reversed, from=1-2, to=1-4]
	\arrow["{\mathfrak{C}_{3}}", tail reversed, from=1-4, to=2-5]
	\arrow["{\mathfrak{C}_{m_{1}}}", tail reversed, from=2-1, to=1-2]
	\arrow["{\mathfrak{C}_{4}}", tail reversed, from=2-5, to=3-3]
	\arrow["{\mathfrak{C}_{n}}", tail reversed, from=3-2, to=2-1]
	\arrow["{\mathfrak{C}_{n-1}}", tail reversed, from=3-3, to=3-2]
\end{tikzcd}\]
holds, then
\begin{equation}\left|\frac{p_{1}-p_{2}}{|p_{1}-p_{2}|}\wedge\cdots\wedge\frac{p_{n}-p_{1}}{|p_{n}-p_{1}|}\right|\geq{}c/2.\label{directions}\end{equation}

Assume that we have selected a set of points so that distinct points obey $|p_{j}-p_{j'}|\geq{}\lambda^{-1+\epsilon}$.They must therefore lie in a $(r-1)$-dimension linear subspace, which contradicts \eqref{directions}. So some of the pairs $(p_{i},p_{i+1})$ must be the same. Therefore we can contract the cycle and we can apply the same argument again. 

\begin{figure}[h!]\label{fig:contract}
\[\begin{tikzcd}
	& {p_{2}=p_{1}} \\
	{p_{1}} && {p_{3}} && {p_{1}} && {p_{3}} \\
	& {p_{r}} &&&& {p_{r}}
	\arrow["{\CC_{1}}", tail reversed, from=2-1, to=1-2]
	\arrow["{\CC_{2}}", tail reversed, from=1-2, to=2-3]
	\arrow["{\CC_{3}}", tail reversed, from=2-3, to=3-2]
	\arrow["{\CC_{4}}", tail reversed, from=3-2, to=2-1]
	\arrow["{\CC_{2}}", tail reversed, from=2-5, to=2-7]
	\arrow["{\CC_{3}}", from=2-7, to=3-6]
	\arrow["{\CC_{4}}", tail reversed, from=3-6, to=2-5]
\end{tikzcd}\]
\caption{In this case ($r=4$) we have $p_{1}=p_{2}$ so we can collapse the loop on the left to that on the right.}

\end{figure}

So the only way that the diagram can hold is if all the points are the same. Therefore this is a uniformly bounded simultaneous system for any $X\subset \R^{d}$ which has the property that if $p,q\in X, |p-q|\leq \lambda^{-1+\epsilon}$ then $p=q$. 

Let us now select some points. Define
\[\left|T_{m}f_{m}(x)\right|\leq  1\]
so let
\begin{equation}\Omega_{i}=\left\{x\middle|\;   2^{-i-1}\leq \prod_{m=1}^{d}\left|T_{m}f_{m}(x)\right|\leq  2^{-i}\right\}\quad 0\leq  i \leq I\label{omegak}\end{equation}
with $I$ such that $2^{-I}=\lambda^{-\frac{d(d-1)}{2}}$. We will use Theorem \ref{thm:simsat} to control the size of $\Omega_{i}$. In particular we will see that for any $\epsilon>0$
\begin{equation}|\Omega_{i}|\leq C\lambda^{-d+\epsilon d}2^{\frac{2(k+1)}{(d-1)(1-\epsilon)}}\label{omegaest}\end{equation}
which for a suitable choice of $\epsilon$, implies \eqref{nearop}.

Let $p_{1}\in\Omega_{i}=\Omega_{i}^{1}$, clearly 
$$2^{-i-1}\leq \prod_{m=1}^{d}\left|T_{m}f_{m}(p_{1})\right|.$$
Now set $\Omega_{i}^{2}=\Omega^{1}_{i}\setminus \mathbb{B}_{\lambda^{-1+\epsilon}}(p_{1})$ and pick $p_{2}\in \Omega_{i}^{2}$, clearly also
$$2^{-i-1}\leq \prod_{m=1}^{d}\left|T_{m}f_{m}(p_{2})\right|.$$
Then, excise $\mathbb{B}_{\lambda^{-1+\epsilon}}(p_{2})$ from $\Omega_{i}^{2}$ and chose $p_{3}$ and so forth. The process terminates at some $N$ when $\Omega_{i}^{N+1}=\emptyset$. Comparing volumes we find that 
$$N\geq{}\frac{|\Omega_{i}|}{\lambda^{-d+\epsilon d}}$$
and of course each $p_{j}$ obeys
$$2^{-i-1}\leq \prod_{m=1}^{d}\left|T_{m}f_{m}(p_{j})\right|.$$
Clearly we have the a priori bound $N\lambda^{d}$ and since the decay rate is $O(\lambda^{-\infty})$  we may apply the results of Theorem \ref{thm:simsat} for any $\epsilon>0$, to obtain
\[N^{1-\epsilon}\leq C 2^{\frac{2(i+1)}{d-1}}\]
and so
\[|\Omega_{i}|\leq C \lambda^{-d+\epsilon d}2^{\frac{2(i+1)}{(d-1)(1-\epsilon)}}.\]
The same sort of arguments produce $k$ multilinear estimates for $p\geq{}\frac{2k}{k-1}$ under just a transversality condition. The proof of Theorem \ref{thm:simsat} makes the critical numerology of $p=\frac{2k}{k-1}$  quite clear. Suppose we are estimating $\prod_{m=1}^{k}\mathcal{E}_{m,\lambda}f_{m}$ over a region $\Omega$ covered by $N$ balls $\mathbb{B}_{\lambda^{-1+\epsilon}}(p_{j})$ so that $Tu(p_{j})$ are all around the same size. Theorem \ref{thm:simsat} we average over all possible loops with $k$ points chosen in a set of $N$ objects and find that the major contribution comes from the trivial loops that only consist of one point. Therefore since the number of loops of $k$ points grows as $N^{k}$ and the number of trivial loops grows as $N$ we expect a $N^{-(k-1)}$ improvement over the trivial estimate in the $TT^{\star}$ argument, corresponding to a $N^{-\frac{k-1}{2}}$ improvement over the trivial estimate for $\prod_{m=1}^{k}\mathcal{E}_{m,\lambda}f_{m}$. Then
\[\norm{\prod_{m=1}^{k}\mathcal{E}_{m,\lambda}f_{m}}_{L^{q}(\Omega)}\lesssim N^{-\frac{d-1}{2}}\lambda^{-\frac{d+d\epsilon}{q}}N^{\frac{1}{q}}.\]
For $q>\frac{2}{k-1}$ the decay from the simultaneous saturation result is enough to counteract the growth in measure. But for $q<\frac{2}{k-1}$ the reverse is true. The critical value $p=\frac{2}{k-1}$ is the value at which the growth and decay are perfectly matched.

The case in which only transversality is assumed is therefore understood, up to questions about endpoints.  What has remained unresolved is the situation where, for $k<d$, additional curvature is present in the “missing’’ directions.  The standing conjecture, for which we shall establish the $\lambda^{\epsilon}$ loss version, asserts that under suitable curvature hypotheses the parameter-dependent $k$-multilinear estimate \eqref{klinconj} holds for all pairs $(p,2)$ satisfying
\[p \ge p(k)= \frac{2(d+k)}{d-k-2}.\]

The conjecture was originally formulated by Foschi and Klainerman \cite{FK00} for $k=2$ on subsets of the cone and paraboloid. Bennett \cite{Ben14} extended the statement to hypersurfaces with everywhere positive curvature and Bejenaru \cite{Bej17} stated the conjecture in its present most general geometric form. There are a number of different formulations of the curvature condition. Heuristically it states that there is curvature in all ``missing'' directions. Here it is convenient to state it in terms of the shape operators of the $\mathcal{H}_{m}$ and the subspace spanned by the set $\{\nu_{1}(\xi^{1}),\dotsm\nu_{m}(\xi^{m})\}$.  Namely that for $\boldsymbol{\xi}=(\xi^{1},\dots,\xi^{m})$, if $\mathcal{Z}_{\boldsymbol{\xi}}$ is the subspace spanned by $\{\nu_{1}(\xi^{1}),\dotsm\nu_{m}(\xi^{m})\}$ and $\gamma_{\mathcal{Z}^{\perp}_{\boldsymbol{\xi}}}$ the restriction to $\mathcal{Z}_{\boldsymbol{\xi}}^{\perp}$ then for each shape operator $S_{m}$; $\gamma^{\star}_{\mathcal{Z}^{\perp}_{\boldsymbol{\xi}}}S\gamma_{\mathcal{Z}^{\perp}_{\boldsymbol{\xi}}}$ is non-degenerate. 

Significant progress toward this conjecture has been made in a sequence of works by Bejenaru.  In \cite{Bej17} and \cite{Bej19}, he established the estimate for certain model hypersurfaces, though these do not include the standard examples such as the paraboloid, sphere, or cone.  Subsequently, Bejenaru \cite{B22} proved the theorem in the case $k=d-1$ for $p>p(k)$. This result settles the conjecture in dimension $d=3$, since the case $k=2$ follows from the existing bilinear theory.

In the present paper we establish the $\lambda^{\epsilon}$ loss bound for arbitrary $k$ under Benjenaru's curvature assumptions.

Our approach is to interpolate between a simultaneous saturation result, which exploits transversality, and a Tomas-Stein type estimate, which exploits curvature. 

Using the transversality result we may associate $\nu_{1},\dots,\nu_{k}$ with the coordinate directions $e_{1},\dots,e_{k}$ in the following fashion
\begin{align*}
e_{1}&=\nu_{1}\\
e_{2}&=\frac{\nu_{2}-\mathfrak{P}_{\{e_{1}\}}(\nu_{2})}{|\nu_{2}-\mathfrak{P}_{\{e_{1}\}}|}\\
&\vdots \\
e_{k}&=\frac{\nu_{k}-\mathfrak{P}_{\{e_{1},\dots,e_{k-1}\}}(\nu_{k})}{|\nu_{k}-\mathfrak{P}_{\{e_{1},\dots,e_{k-1}\}}|}
\end{align*}
where $\mathfrak{P}_{\{e_{1},\dots,e_{n}\}}(v)$ is the projection of $v$ onto the subspace spanned by $\{e_{1},\dots,e_{n}\}$. With this orientation the transversality condition 
\[\left|\nu_{1}\wedge\cdots\wedge\nu_{k}\right|\geq{}c\]
ensures that
\[\left|\mathfrak{P}_{\{e_{1},\dots,e_{k}\}}(\nu_{1})\wedge\cdots\wedge \mathfrak{P}_{\{e_{1},\dots,e_{k}\}}(\nu_{1})\right|\geq\tilde{c}\]
where $\tilde{c}$ depends on $c$ and the dimension alone. 

Then we
parameterise the hypersurfaces as
\begin{align*}
&H_{1}=\{\xi\mid \xi_{1}=\Sigma_{1}(\xi_{2},\dots,\xi_{d})\}\\
& \vdots\\
& H_{k}=\{\xi\mid \xi_{k}=\Sigma_{k}(\xi_{1},\dots,\xi_{k-1},\xi_{k+1},\dots,\xi_{d}).\}\end{align*}
Accordingly we adopt the notation that $x=(y,z)$ where $y\in\R^{k}$ and $z\in \R^{n-k}$. We make a similar separation of the dual variables writing $(\bar{\xi},\eta)$ where $\bar{\xi}\in \R^{k-1}$ and $\eta\in \R^{d-k}$. 

In what follows we make extensive use of mixed $L^{p}$ norms. These function spaces, first introduced by Benedek and Panzone \cite{BP61}, have several notational conventions. We adopt here the iterated mixed-norm notation, with the relevant variables indicated by subscripts. This choice agrees with that most commonly employed in the literature on Strichartz spaces, from which many of our ideas will be drawn. In this notation,
\[\norm{g}_{L^{p}_{y}L^{q}_{z}}
=\left(\int \left(\int |g(y,z)|^{q},dz\right)^{p/q} dy\right)^{1/p}.\]

As with the usual Lebesgue spaces, one may also define mixed-norm spaces for exponents $0 < p <1$. Although these spaces are not complete, they arise naturally in the study of multilinear restriction and extension operators. Interpolation between mixed Lebesgue spaces proceeds as in the classical case \cite{BP61}, with the qualification that when any exponent is below one, the interpolation result holds only for simple functions. This caveat does not cause us any difficulties. Since we work on compact regions, any function can be approximated in $L^{\infty}$ by a sequence of simple functions, and hence in any mixed Lebesgue norm.

We first use a simultaneous saturation argument to establish  the bound
\[
\norm{\prod_{m=1}^{k}\mathcal{E}_{m,\lambda}f_{m}}_{L^{\frac{2}{k-1}}_{y}L^{\frac{2}{k}}_{z}}\lesssim \lambda^{-\frac{k(d-1)}{2}+\delta }\prod_{m=1}^{k}\norm{f_{m}}_{L^{2}(H_{m})}.\]
Our goal is to locate $\frac{p(k)}{k}$ within an interpolation scale where one of the interpolation endpoints is $L^{\frac{2}{k-1}}_{y}L^{\frac{2}{k}}_{z}$. To that end, we seek a pairs $(r_{1},r_{2})$ where $\frac{2}{k}\le r_{1}<p(k)$ and $p(k)<r_{2}\le \infty$ such that $L^{\frac{p(k)}{k}}_{y}L^{\frac{p(k)}{k}}_{z}$ appears as a interpolated space between $L^{\frac{2}{k-1}}_{y}L^{\frac{2}{k}}_{z}$ and $L^{r_{1}}_{y}L^{r_{2}}_{z}$. Therefore we require that there is a $\theta\in(0,1)$ so that
\[\frac{k}{p(k)}=\frac{\theta(k-1)}{2}+\frac{(1-\theta)}{r_{1}},
\qquad
\frac{k}{p(k)}=\frac{\theta k}{2}+\frac{1-\theta}{r_{2}}.\]
Several possible choices of $(r_{1},r_{2},\theta)$ are available, but it will become clear during the technical development that setting $r_{1}=2/k$ is advantageous. In this case, the interpolation relations give
\[\theta=\frac{2k}{d+k},\qquad
r_{2}=\frac{2(d-k)}{k(d-k-2)},\]
which defines a valid interpolation scale whenever $k\le d-2$.

The case where $k=d-1$ is established by Bejenaru in \cite{B22} so it is not strictly necessary to treat it here. However to complete the picture (and because it requires almost no extra work), we fix $r_{2}=\infty$ and set
\[\theta=\frac{2d-3}{2d-1},\qquad r_{1}=\frac{4}{2d-3},
\]
In Propositions \ref{prop:TSbound2} and \ref{prop:TSboundinf}, we adapt the classical Tomas–Stein $L^{2}\to L^{p}$ estimates to establish the corresponding mixed-norm bounds under the curvature assumption.

We now turn to the proof of these estimates.

\begin{prop}\label{prop:klintrans}
\begin{equation}
\norm{\prod_{m=1}^{k}\mathcal{E}_{m,\lambda}f_{m}}_{L^{\frac{2}{k-1}}_{y}L^{\frac{2}{k}}_{z}}\lesssim \lambda^{-\frac{k(d-1)}{2}+\delta }\prod_{m=1}^{k}\norm{f_{m}}_{L^{2}(H_{m})}\label{ssinterp}\end{equation}
\end{prop}

\begin{proof}
First fix $y$. Then the $L^{\frac{2}{k}}_{z}$ estimate can be reduced (via H\"{o}lder) to $k$ $L^{2}$ estimates,
i.e. for every $y$,
\begin{align*}\left(\int_{z}\left|\prod_{m=1}^{k}\mathcal{E}_{m,\lambda}f_{m}\Big|_{(y,z)}\right|^{\frac{2}{k}}dz\right)^{\frac{k}{2}}&\leq \prod_{m=1}^{k}\left(\int_{z}\left|\mathcal{E}_{m,\lambda}f_{m}\Big|_{(y,z)}\right|^{2}dz\right)^{\frac{1}{2}}\\
&=\prod_{m=1}^{k}\norm{\mathcal{E}_{m,\lambda}f_{m}(y,\cdot)}_{L^{2}_{z}}\\
&=\prod_{m=1}^{k}\sup_{\substack{\phi_{y,m}\in L^{2}_{z}\\ \norm{\phi_{y,m}}_{L^{2}_{z}}=1}}\left|\langle \phi_{y,m},\mathcal{E}_{m,\lambda}f_{m}(y,\cdot)\rangle_{z}\right|\end{align*}
Accordingly let $\phi_{y,m}\in L^{2}_{z}$ with $\norm{\phi_{y,m}}_{L^{2}_{z}}=1$
and define the operator $T_{m,\phi}:L^{2}(H_{1})\to  X^{\C}$
by
\[T_{m,\phi}f(y)=\int_{z}\int_{H_{m}} e^{i\lambda\langle x,\xi\rangle}\phi_{y,m}(z)f(\xi)d\mu_{m}(\xi)\]
where we have written $x=(y,z)$. We assume that $X\subset\R^{k}$ with a separation property. That is if $p,q\in X$ then $|p-q|\geq\lambda^{-1+\epsilon}$. For fixed $p$
\[\left|T_{m,\phi_{m,p}}(p)f\right|\leq \norm{\phi_{p,m}}_{L^{2}}\norm{\mathcal{E}_{m,\lambda}f\Big|_{(p,\cdot)}}_{L^{2}_{z}}\] and adopting the parametrisation for $H_{m}$
\[\norm{\mathcal{E}_{m,\lambda}f\Big|_{(p,\cdot)}}_{L^{2}_{z}}^{2}=\int e^{i\lambda\langle z,\eta-\tilde{\eta}\rangle}f(\eta)\overline{g(\tilde{\eta})}dz d\eta d\tilde{\eta}\]
Clearly there are no critical points, $z$, for the phase function unless $\eta=\tilde{\eta}$. Therefore integration by parts combined with Shur's test yield
\[\norm{\mathcal{E}_{m,\lambda}f\Big|_{(p,\cdot)}}_{L^{2}_{z}}^{2}\leq \lambda^{-(d-k)}\norm{f}_{L^{2}(H_{m})}^{2}.\]
So
\begin{equation}\left|T_{m,\phi_{p,m}}(p)f\right|\leq \lambda^{-\frac{d-k}{2}}.\label{kuniformbound}\end{equation}
Therefore for any $p,q$
\[\left|T_{m,\phi_{p,m}}(p)\left(T_{m,\phi_{q,m}}(q)\right)^{\star}\right|\leq \lambda^{-(d-k)}.\]
That is, in this case $B=\lambda^{-(d-k)}$. Now let's see the decay conditions. We compute
\begin{multline*}T_{m,\phi_{p,m}}(p)\left(T_{m,\phi_{p,m}}(q)\right)^{\star}=\int_{H_{m}} e^{i\lambda \left(\langle \bar{p}-\bar{q},\bar{\xi}\rangle +\langle z-w,\eta\rangle+(p_{m}-q_{m})\Sigma_{m}(\bar{\xi},\eta)\right)}\\
\times b(\bar{\xi},\eta)\phi_{p,m}(z)\phi_{q,m}(w)d\bar{\xi}d\eta dzdw\end{multline*}
 We have already seen that the integral in $(\bar{\xi},\eta)$ experiences rapid decay if $(p-q,z-w)$ does not point in the direction of the normal of $H_{m}$. To set up our conditions we introduce the notation $\SS_{p}$,
\[\SS_{p}=\{(p,z) \mid z\in \R^{d-k}\}\]
That is $\SS_{p}$ is the translation of the $y=0$ subspace to $p$. If there are no points $x(l),x(r)$ so that $x(l)\in \SS_{p}$ and $x(r)\in\SS_{q}$ with
\[\left|\frac{x(l)-x(r)}{|x(l)-x(r)|}-(\pm \nu_{m})\right|\leq 2\tilde{\epsilon}\]
then we can be sure that $|T_{m,\phi_{p,m}}(p)\left(T_{m,\phi_{p,m}}(q)\right)^{\star}|$ experiences rapid decay of rate $O(\lambda^{-\infty})$. Therefore our condition on $p$ and $q$ becomes
\[\CC_{m}=\left\{(p,q)\in X^{2}\middle|\substack{p=q\text{ or }\\ \exists x(l)\in \SS_{p}, x(r)\in\SS_{q}\\
\left|\frac{x(l)-x(r)}{|x(l)-x(r)|}-(\pm \nu_{m})\right|\leq 2\tilde{\epsilon}}\right\}.\]
Now consider the diagram 
\[\begin{tikzcd}
	& {p_{2}} && { p_{3}} \\
	{ p_{1}} &&&& { p_{4}} \\
	& {p_{n}} & \cdots & {}
	\arrow["{\mathfrak{C}_{m_{2}}}", tail reversed, from=1-2, to=1-4]
	\arrow["{\mathfrak{C}_{3}}", tail reversed, from=1-4, to=2-5]
	\arrow["{\mathfrak{C}_{m_{1}}}", tail reversed, from=2-1, to=1-2]
	\arrow["{\mathfrak{C}_{4}}", tail reversed, from=2-5, to=3-3]
	\arrow["{\mathfrak{C}_{n}}", tail reversed, from=3-2, to=2-1]
	\arrow["{\mathfrak{C}_{n-1}}", tail reversed, from=3-3, to=3-2]
\end{tikzcd}\]
and assume that it holds but the $p_{i}$ are all distinct. Therefore for each $m_{i}$ there is an $x_{i}(l)\in \SS_{p_{m_{i}}}$ and $x_{i}(r)\in \SS_{p_{m_{i+1}}}$ such that
\[\left|\frac{x_{i}(l)-x_{i}(r)}{|x_{i}(l)-x_{i}(r)|}-(\pm\nu_{m})\right|\leq \tilde{\epsilon}.\]
On the other hand $x_{i}(r)$ and $x_{i+1}(l)$ both lie in $\SS_{p_{m_{i+1}}}$. So we can create a new diagram involving only the $x_{i}(r)$
\[\begin{tikzcd}
	& {x_{2}(r)} && { x_{3}(r)} \\
	{ x_{1}(r)} &&&& { x_{4}(r)} \\
	& {x_{n}(r)} & \cdots & {}
	\arrow["{\tilde{\mathfrak{C}}_{m_{2}}}", tail reversed, from=1-2, to=1-4]
	\arrow["{\tilde{\mathfrak{C}}_{3}}", tail reversed, from=1-4, to=2-5]
	\arrow["{\tilde{\mathfrak{C}}_{m_{1}}}", tail reversed, from=2-1, to=1-2]
	\arrow["{\tilde{\mathfrak{C}}_{4}}", tail reversed, from=2-5, to=3-3]
	\arrow["{\tilde{\mathfrak{C}}_{n}}", tail reversed, from=3-2, to=2-1]
	\arrow["{\tilde{\mathfrak{C}}_{n-1}}", tail reversed, from=3-3, to=3-2]
\end{tikzcd}\]
where
\[\tilde{\CC}_{m}=\left\{(x_{1},x_{2})\in (X\times\R^{d-k})^2\middle| \substack{x_{1}=x_{2}\text{ or }\\x_{1}-x_{2}=v+w\\
\left|\frac{v}{|v|}-(\pm \nu_{m})\right|\leq \tilde{\epsilon}\\
\langle w,e_{l}\rangle =0 \;l=1,\dots,k}\right\} \]

which if the $x_{i}(r)$ are distinct would require
\[\left|\mathfrak{P}_{\{e_{1},\dots,e_{k}\}}(\nu_{1})\wedge\cdots\wedge\mathfrak{P}_{\{e_{1},\dots,e_{k}\}}(\nu_{k})\right|\leq C\tilde{\epsilon}.\]
For a suitable choice of $\tilde{\epsilon}$ this contradicts the transversality condition
\[\left|\mathfrak{P}_{\{e_{1},\dots,e_{k}\}}(\nu_{1})\wedge\cdots\wedge\mathfrak{P}_{\{e_{1},\dots,e_{k}\}}(\nu_{k})\right|\geq \tilde{c}.\]
If any pairs of points are the same we can (as in the $k=d$ case) contract the loop and make the same argument. So all the points in the loop must be the same. So we do indeed have a uniformly bounded simultaneous system. All that is required is to locate the points. We need to estimate
\[\left(\int_{y}\left(\prod_{m=1}^{k}\norm{\mathcal{E}_{m,\lambda}f_{m}\Big|_{(y,\cdot)}}_{L^{2}_{z}}\right)^{\frac{2}{k-1}}dy\right)^{\frac{k-1}{2}}\]
for $\norm{f_{m}}_{L^{2}(H_{m})}=1$ (for convenience we will henceforth assume that we are so normalised) The bound \eqref{kuniformbound} ensures that for any $y$
\[\prod_{m=1}^{k}\norm{\mathcal{E}_{m,\lambda}f_{m}\Big|_{(y,\cdot)}}_{L^{2}_{z}}\leq \lambda^{-\frac{k(d-k)}{2}}.\]
As in the $k=d$ case set
\[\Omega_{i}=\left\{y\in\R^{k}\middle| 2^{-i-1}\lambda^{-\frac{k(d-k)}{2}}\leq\prod_{m=1}^{k}\norm{\mathcal{E}_{m,\lambda}f_{m}\Big|_{(y,\cdot)}}_{L^{2}_{z}}\leq 2^{-i}\lambda^{-\frac{k(d-k)}{2}}\right\}\]
where $i=0,\dots, I$ so that $2^{-I}=\lambda^{-\frac{k(k-1)}{2}}$. For $\mathbb{B}_{1}\setminus \bigcup_{i=1}^{I}\Omega_{i}$ the estimate \eqref{ssinterp} is trivial. We the find the points $p_{1},\dots,p_{N}$ in the same fashion as the $k=d$ case, so that each $p_{j}\in\Omega_{i}$ and distinct points are at least $\lambda^{-1+\epsilon}$ apart. Again this gives us
\[N\geq{}\frac{|\Omega_{i}|}{\lambda^{-k+k\epsilon}}\]
Having found the points we can then locate $\phi_{p,m}$ so that
\[\left|\prod_{m=1}^{k}T_{m,\phi_{p,m}}f_{j}\Big|_{p_{j}}\right|\geq 2^{-i-1}\lambda^{-\frac{k(d-k)}{2}}\]
so $L=2^{-\frac{i+1}{k}}\lambda^{-\frac{d-k}{2}}$. Recall that $B=\lambda^{-\frac{d-k}{2}}$. So applying Theorem \ref{thm:simsat} we obtain
\[|\Omega_{i}|\leq N\lambda^{-k+k\epsilon}\leq C2^{2(i+1)k}{(k-1)(1-\epsilon)}\lambda^{-k+k\epsilon}\]
which for a suitable choice of $\epsilon$ gives
\[|\Omega_{i}|\leq C2^{\frac{2ki}{k-1}}\lambda^{-k+\frac{2\delta}{k-1}}\]
giving \eqref{ssinterp}.
\end{proof}

We now consider the opposite end of the interpolation scale. At this point, H\"{o}lder’s inequality reduces the argument to the analysis of $k$ linear-type estimates. Our approach follows the general framework of the Tomas–Stein, \cite{T79} and \cite{S}, $L^{2}\to L^{p}$ restriction theorem, adapted here to a mixed-norm setting. The central observation, originating in that work, is that one of the spatial variables may be regarded as a time parameter, permitting the formulation of an associated Strichartz estimate. Subsequent developments have refined these ideas both within the $L^{2}\to L^{p}$ restriction/extension theory and in the closely related field of $L^{p}$ concentration phenomena for eigenfunctions.

Before proceeding to the proof, we outline the general structure of the argument. Although the material is not new, it is included here for completeness. The abstract Strichartz formulation follows Keel–Tao \cite{KT98}, while the parameter-dependent version was developed by Koch–Tataru–Zworski \cite{KTZ07} for unitary operators and by Tacy \cite{T10} for the non-unitary case. Suppose that $T(t)$ is a one parameter family of operators $T(t):L^{2}\to L^{p}$. 
The operator $\norm{T}_{L^{p}}$ is viewed as an $L^{p}_{t}L^{p}_{x}$ Strichartz estimate. Suppose that the following bounds hold:
\begin{equation}
\norm{T(t)T^{\star}(s)}_{L^{1}\to L^{\infty}}
\leq \lambda^{-\mu_{\infty}}\bigl(\lambda^{-1}+|t-s|\bigr)^{-\sigma_{\infty}},
\label{Sinftest}
\end{equation}
and
\begin{equation}
\norm{T(t)T^{\star}(s)}_{L^{2}\to L^{2}}
\leq \lambda^{-\mu_{2}}\bigl(\lambda^{-1}+|t-s|\bigr)^{-\sigma_{2}}.
\label{S2est}
\end{equation}
Then since 
\[
\norm{T(\cdot)}_{L^{q}_{t}L^{p}_{x}}^{2}
=\sup_{\norm{F}=\norm{G}=1}
\left|\iint \langle T^{\star}(s)F(s),T^{\star}(t)G(t)\rangle,ds,dt\right|
\]
we can control $\norm{T(\cdot)}_{L^{q}_{t}L^{p}_{x}}$ by controlling the bilinear form. Interpolating between the bilinear form estimates corresponding to \eqref{Sinftest} and \eqref{S2est} gives
\[
\left|\langle T^{\star}(s)F(s),T^{\star}(t)G(t)\rangle\right|
\leq \lambda^{-\mu_{p}}\left(\lambda^{-1}+|t-s|\right)^{-\sigma_{p}}
\norm{F}_{p'}\norm{G}_{p'}.
\]
The $(t,s)$-integration is then resolved by Young’s inequality when
\[\frac{q}{2}\sigma_{p}\neq -1,\]
or by the Hardy–Littlewood–Sobolev inequality when
\[\frac{q}{2}\sigma_{p}=-1,\qquad q\neq 2.\]
If $q=2$ and $\sigma_{p}=-1$, a local-in-time estimate with logarithmic loss remains available. The quantities in \eqref{Sinftest} and \eqref{S2est} therefore determine the entire numerology of the argument, and the sharp bounds are typically achieved in the critical case $\frac{q}{2}\sigma_{p}=-1$.

\begin{prop}\label{prop:TSbound2}
Suppose that for each $m$ the hypersurface $\zeta=\Sigma_{m}(0,\eta)$ sitting in $\R^{d-k+1}$ is curved. Then
\begin{equation}
\norm{\prod_{m=1}^{k}\mathcal{E}_{m}f_{m}}_{L^{\frac{2}{k}}_{y}L^{\frac{2(d-k)}{k(d-k-2)}}_{z}}\lesssim \lambda^{-\frac{k(d-2)}{2}+3\epsilon}\prod_{m=1}^{k}\norm{f_{m}}_{L^{2}(H_{m})}\quad k=1,\dots,d-2.\label{curv2}\end{equation}

\end{prop}

\begin{proof}
 First we use H\"{o}lder's inequality to see that 
 \[\norm{\prod_{m=1}^{k}\mathcal{E}_{m,\lambda}f_{m}}_{L^{\frac{2}{k}}_{y}L^{\frac{2(d-k)}{k(d-k-2)}}_{z}}\leq \prod_{m=1}^{k}\norm{\mathcal{E}_{m}f}_{L_{y}^{2}L_{z}^{\frac{2(d-k)}{d-k-2}}}\]
 so it would be sufficient to prove that for each $m$
\begin{equation}
 \norm{\mathcal{E}_{m,\lambda}f}_{L_{y}^{2}L_{z}^{\frac{2(d-k)}{d-k-2}}}\lesssim \lambda^{-\frac{d-2}{2}}\norm{f_{m}}_{L^{2}}.\label{lcurve2}\end{equation}
 We will establish this for $m=1$, all other cases can be converted to this one by a rotation of the coordinate system. We use the parametrisation $\xi_{1}=\Sigma_{1}(\bar{\xi},\eta)$ where $\xi=(\xi_{1},\bar{\xi},\eta)$ and $x=(t,\bar{y},z)$. In this notation
\[\mathcal{E}_{1,\lambda}f=\int e^{i\lambda\left(\langle \bar{y},\bar{\xi}+\langle z,\eta\rangle+t\Sigma_{1}(\bar{\xi},\eta)\right)}b(\bar{\xi},\eta)f(\bar{\xi},\eta)d\bar{\xi}d\eta.\]
First we will decompose in the size of $\nabla_{\bar{\xi}}\Sigma_{1}$, this scale tells us how flat the surface $\zeta=\Sigma_{1}(\bar{\xi},\eta)$ is with respect to the $\bar{\xi}$ variables.
We write
\[\mathcal{E}_{1,\lambda}=T_{0}+\sum_{r=1}^{R}T_{r}\]
where
\[T_{0}=\int e^{i\lambda\left(\langle \bar{y},\bar{\xi}\rangle+\langle z,\eta\rangle+t\Sigma_{1}(\bar{\xi},\eta)\right)}b(\bar{\xi},\eta)\chi_{0}(\nabla_{\bar{\xi}}\Sigma_{1}(\bar{\xi},\eta))f(\bar{\xi},\eta)d\bar{\xi}d\eta,\]
\[T_{r}=\int e^{i\lambda\left(\langle \bar{y},\bar{\xi}\rangle+\langle z,\eta\rangle+t\Sigma_{1}(\bar{\xi},\eta)\right)}b(\bar{\xi},\eta)\chi_{r}(\nabla_{\bar{\xi}}\Sigma_{1}(\bar{\xi},\eta))f(\bar{\xi},\eta)d\bar{\xi}d\eta,\]
with $\chi_{0}$  supported in the ball $\B_{\lambda^{-1+2\epsilon}}(0)$, $\chi_{r}$ is supported in the annulus $A_{[\lambda^{-1+r\epsilon};\lambda^{-1+(r+2)\epsilon}]}$, 
\[|D^{\beta}_{\bar{\zeta}}\chi_{r}(\bar{\zeta)}|\leq \lambda^{(1-r\epsilon)|\beta|}\]
and
\[1=\sum_{r=0}^{R}\chi_{r}(\bar{\zeta}).\]
Note that since we are working on a region of compact support (in $\bar{\xi}$) the sum is finite and $R$ is independent of $\lambda$ ($R\approx \frac{1}{\epsilon}$). Let us first treat the cases $r\geq{}1$. Here there is a lower bound for $\nabla_{\bar{\xi}}\Sigma_{1}$. Therefore there is some $i$ (which we will assume to be $2$) so that,
\[|\partial_{\bar{\xi}_{2}}\Sigma_{1}(\bar{\xi},\eta)|\geq{}c\lambda^{-1+\epsilon r}.\]
We make the change of variables
\[\Sigma_{1}(\bar{\xi},\eta)-\Sigma_{1}(0,\eta)\mapsto \lambda^{-1+\epsilon r}\bar{\xi}_{2}\]
to obtain
\[T_{r}=\int e^{i\lambda\left(\langle \tilde{y},\tilde{\xi}\rangle+y_{2}\psi(\bar{\xi},\eta)+\langle z,\eta\rangle+t\Sigma_{1}(0,\eta)+\lambda^{-1+\epsilon r}t\bar{\xi_{2}}\right)}
\tilde{b}(\bar{\xi},\eta)f(\psi(\bar{\xi},\eta),\eta)d\bar{\xi}d\eta\]
where $|\partial_{\bar{\xi}_{2}}\psi(\bar{\xi},\eta)|\geq{}c>0$. To enable us to factorise $T_{r}$ and treat the inner and outer $L^{p}$ norms separately we use a semiclassical resolution of the identity to introduce integrated variables $(\rho,\tau)$. Later we will tune the $(\tau,\rho)$ localisation dependent on $r$.  We claim that
\begin{multline}T_{r}=\frac{\lambda^{\epsilon r}}{2\pi}\int e^{i\lambda\left(\langle \tilde{y},\tilde{\xi}\rangle+y_{2}\psi(\bar{\xi},\eta)+\langle z,\eta\rangle+t\Sigma_{1}(0,\eta)\right)+i\lambda^{\epsilon r}\left(t\rho-\tau(\rho-\bar{\xi}_{2})\right)}\\
\times\tilde{b}(\bar{\xi},\eta)a(\tau)a(\rho)f(\psi(\bar{\xi},\eta),\eta)d\bar{\xi}d\eta d\tau d\rho\label{integratedvariables}\end{multline}
for some $a$ smooth and equal to one in $[-1,1]$ supported in $[-2,2]$. To see \eqref{integratedvariables} holds perform the $(\tau,\rho)$ integral via the method of stationary phase. There is a single non-degenerate critical point at
\[\rho=\bar{\xi}_{2}\quad t=\tau\]
and so \eqref{integratedvariables} follows from the stationary phase formula, with parameter $\lambda^{\epsilon r}$ in dimension $2$. 
If $|t-\tau|\geq \lambda^{-\epsilon (r-1)}$ the phase is non-stationary in $\rho$ and we may integrate by parts to obtain a $O(\lambda^{-\infty})$ contribution we we write
\begin{multline*}T_{r}f=\frac{\lambda^{\epsilon r}}{2\pi}\int e^{i\lambda\left(\langle \tilde{y},\tilde{\xi}\rangle+y_{2}\psi(\bar{\xi},\eta)+\langle z,\eta\rangle+t\Sigma_{1}(0,\eta)\right)+\lambda^{\epsilon r}\left(t\rho-\tau(\rho-\bar{\xi}_{2})\right)}\\
\times \tilde{b}(\bar{\xi},\eta)a(\tau)a(\rho)a(\lambda^{\epsilon (r-1)}(t-\tau))\\
\times f(\psi(\bar{\xi},\eta),\eta)d\bar{\xi}d\eta d\tau d\rho+O(\lambda^{-\infty}).\end{multline*}
We re-write this as
\[T_{r}f=I_{r}F(\bar{y},\cdot)\]
where
\[I_{r}g=\frac{\lambda^{\frac{\epsilon r}{2}}}{2\pi}\int e^{i\lambda\left(\langle z,\eta\rangle+t\Sigma_{1}(0,\eta)\right)+\lambda^{\epsilon r}t\rho)}
a(\tau)a(\lambda^{\epsilon (r-1)}(t-\tau))a(\rho)g(\tau,\rho,\eta)d\eta d\tau d\rho\]
and
\begin{multline*}
    F(\bar{y},\rho,\tau,\eta)=\lambda^{\frac{\epsilon r}{2}}\int e^{i\lambda\left(\langle \tilde{y},\tilde{\xi}\rangle+\bar{y}_{2}\psi(\bar{\xi}_{2},\tilde{\xi},\eta)\right)+\lambda^{\epsilon r}\tau(\rho-\bar{\xi}_{2})}\\
    \times \tilde{b}(\bar{\xi}_{2},\tilde{\xi},\eta)a\left(\frac{\rho}{2}\right)f(\psi(\bar{\xi_{2}},\tilde{\xi},\eta),\tilde{\xi},\eta)d\bar{\xi}_{2}d\tilde{\xi}.\end{multline*}
Since $I_{r}$ is local, in the sense that its kernel is zero if $|t-\tau|\geq{}2\lambda^{-\epsilon (r-1)}$ we may work locally on intervals of size $2\lambda^{-\epsilon (r-1)}$ in $t$ and $\tau$. So if we could show that
\begin{equation}\norm{I_{r}}_{L^{2}_{\rho}L^{2}_{\tau}([0,2\lambda^{-\epsilon (r-1)}])L^{2}_{\eta}\to L^{2}_{t}([0,2\lambda^{-\epsilon (r-1)}])L^{\frac{2(d-k)}{d-k-2}}_{z}}\lesssim \lambda^{-\frac{d-k-1}{2}+\frac{\epsilon}{2}}\log(\lambda)\label{innernormbnds}\end{equation}
and
\begin{equation}\norm{F}_{L^{2}_{\bar{y}}L^{2}_{\rho}L_{\tau}^{2}([0,2\lambda^{-\epsilon r}])L^{2}_{\eta}}\lesssim \lambda^{-\frac{k-1}{2}+\frac{\epsilon}{2}}\norm{f}_{L^{2}_{\bar{\xi}}L^{2}_{\eta}}\label{Fbounds}\end{equation}
we could conclude that
\[\norm{T_{r}}_{L^{2}_{\bar{\xi}}L^{2}_{\eta}\to L^{2}_{\bar{y}}L^{2}_{t}L^{\frac{2(d-k)}{d-k-2}}}\lesssim \lambda^{-\frac{d-2}{2}+2\epsilon}\]
The bound for $F$ can be computed directly.
\begin{multline*}\norm{F}_{L^{2}_{\bar{y}}L^{2}_{\rho}L_{\tau}^{2}([0,2\lambda^{-\epsilon r}])L^{2}_{\eta}}^{2}=\lambda^{\epsilon r}\int e^{i\lambda \left(\langle \tilde{y},\tilde{\xi}-\tilde{\zeta}\rangle+\bar{y}_{2}(\psi(\bar{\xi}_{2},\tilde{\xi},\eta)-\psi(\bar{\zeta}_{2},\tilde{\zeta},\eta))\right)+\lambda^{\epsilon r}\tau(\bar{\xi}_{2}-\bar{\zeta}_{2})}\\
\times B(\bar{\xi}_{2},\bar{\zeta}_{2},\tilde{\xi},\eta)a(\lambda^{\epsilon (r-1)}\tau)a\left(\frac{\rho}{2}\right)f(\psi(\bar{\xi_{2}},\tilde{\xi},\eta),\tilde{\xi},\eta)\\
\times \overline{f(\psi(\bar{\zeta_{2}},\tilde{\zeta},\eta),\tilde{\zeta},\eta)}d\bar{\xi}_{2}d\bar{\zeta_{2}}d\bar{y}_{2}d\tilde{y}d\rho d\tilde{\zeta} d\tilde{\xi} d\eta.\end{multline*}
Examining the critical points in $\tilde{y}$ we see that these occur only on the diagonal $\tilde{\xi}=\tilde{\zeta}$. In $\bar{y}_{2}$
the phase is only critical when 
\[\psi(\bar{\xi}_{2},\tilde{\xi},\eta)-\psi(\bar{\zeta}_{2},\tilde{\zeta},\eta)=0.\]
On the support of $B$ recall that $|\partial_{\bar{\xi}_{2}}\psi|\geq{c}$ so when $|\bar{\xi}-\bar{\zeta}|\geq{}\lambda^{-1}$ we can integrate by parts for rapid decay. Therefore and application of Schur's test gives \eqref{Fbounds}. It then only remains to prove the mapping norms for $I_{r}$. If we freeze the $\tau$ variable  and define
\[I_{r}(\tau)g=\int e^{i\lambda\left(\langle z,\eta\rangle+t\Sigma_{1}(0,\eta)\right)+t\lambda^{\epsilon r}\rho}
a(\tau)a(\lambda^{\epsilon (r-1)}(t-\tau))a(\rho)g(\tau,\rho,\eta)d\eta  d\rho\]
it is sufficient to show that 
\[\norm{I_{r}(\tau)}_{L^{2}_{\eta}\to L^{2}_{t}([0,\lambda^{-\epsilon r}])L^{p}_{z}}\lesssim \lambda^{-\frac{d-k-1}{2}}\log(\lambda).\]
So we proceed to treat $I_{\lambda}(\tau)$ with via a classical $(2,p)$ Strichartz estimate. Since it is associated with a restriction extension operator (for the surface $\zeta=\Sigma_{1}(0,\eta)$ in dimension $d-k+1$ for any $t$,
\[\norm{I_{r}(\tau,t)}_{L^{2}_{\eta}\to L^{2}_{z}}\lesssim \lambda^{-\frac{d-k}{2}}\]
so
\[\norm{I_{r}(\tau,t)I^{r}_{r}(\tau,s)}_{L^{2}_{w}\to L^{2}_{z}}\lesssim \lambda^{-(d-k)}.\]
To obtain the $L^{\infty}$ estimate we write
\[I_{r}(\tau,t)I^{\star}_{r}(\tau,s)g=\int e^{i\lambda\left(\langle z-w,\eta\rangle+(t-s)\Sigma_{1}(0,\eta)\right)+\lambda^{\epsilon r}(t-s)\rho}a^{2}(\rho)g(w)dwd\eta d\rho.\]
 We can compute the $\eta$ integral via the method of stationary phase. There is a critical point when
\begin{equation}
z-w=-(t-s)\nabla_{\eta}\Sigma_{1}(0,\eta),\label{etacrit}\end{equation}
so if $|z-w|\geq{}C|t-s|$ the phase is non-stationary and integration by parts ensures rapid decay. The Hessian is given by
\[-(t-s)\partial^{2}_{\eta}\Sigma_{1}(0,\eta)\]
and the curvature condition therefore ensures that
\[\text{det(Hess)}\geq{}|t-s|^{d-k}.\]
So away from $t=s$ the stationary points are non-degenerate. Therefore
\[\norm{I_{r}(\tau,t)I^{\star}_{r}(\tau,s)}_{L^{1}_{w}\to L^{\infty}_{z}} \lesssim \lambda^{-\frac{d-k}{2}}\left(\lambda^{-1}+|t-s|\right)^{-\frac{d-k}{2}}.\]
Interpolating we obtain
\[\norm{I_{r}(\tau,t)I^{\star}_{r}(\tau,t)}_{L^{p'}_{w}\to L^{p}_{z}}\lesssim \lambda^{-\frac{d-k}{2}\left(1+\frac{2}{p}\right)}(\lambda^{-1}+|t-s|)^{-\frac{d-k}{2}\left(1-\frac{2}{p}\right)},\]
and at $p=\frac{2(d-k)}{d-k-2}$
\[\norm{I_{r}(\tau,t)I^{\star}_{r}(\tau,s)}_{L^{\frac{2(d-k)}{d-k+2}}_{w}\to L^{\frac{2(d-k)}{d-k-2}}_{z}}\lesssim\lambda^{-(d-k-1)}(\lambda^{-1}+|t-s|)^{-1}.\]
Note that the power $-1$ is exactly the critical power to resolve the $|t-s|$ for an $L^{2}$ estimate. Since $t$ is restricted to the ball of radius $\lambda^{-\epsilon r}$ we can use Young's inequality and concede a $\log(\lambda)$ loss. That is
\[\norm{I_{r}(\tau,t)I^{\star}_{r}(\tau,s)}_{L^{\frac{2(d-k)}{d-k+2}}_{w}\to L^{\frac{2(d-k)}{d-k-2}}_{z}}\lesssim\lambda^{-(d-k-1)}\log(\lambda)\]
completing the proof of Proposition \ref{prop:TSbound2} for the cases $r\geq{}1$. Notice that for the estimate on $I_{r}$ we never used that $r\neq 0$ so this holds result also in the $r=0$ case. 

Let's complete the treatment of the $r=0$ case.  In this case the surface $\zeta=\Sigma_{1}(\bar{\xi},\eta)$ is essentially ``flat'' in the $\bar{\xi}$ coordinates. Again we will introduce $(\tau,\rho)$ integrated coordinates but in a slightly different fashion. We claim that
\begin{multline*}T_{0}=\frac{\lambda}{2\pi}\int e^{i\lambda\left(\langle \bar{y},\bar{\xi}\rangle+\langle z,\eta\rangle+t\Sigma_{1}(0,\eta)+(t-\tau)\rho+\tau(\Sigma_{1}(\bar{\xi},\eta)-\Sigma_{1}(0,\eta))\right)}\\
\times b(\bar{\xi},\eta)\chi_{0}(\nabla_{\bar{\xi}}\Sigma_{1}(\bar{\xi},\eta))a(\tau)a(\rho)f(\bar{\xi},\eta)d\bar{\xi}d\eta.\end{multline*}
Again this claim follows from an application of the stationary phase formula. The critical points for $(\tau,\rho)$ are at
\[\rho=\Sigma_{1}(\bar{\xi},\eta)-\Sigma_{1}(0,\eta)\quad t=\tau.\]
The localisation in $\nabla_{\bar{\xi}}\Sigma_{1}$ gives us that
\[|\Sigma_{1}(\bar{\xi},\eta)-\Sigma_{1}(0,\eta)|\leq \lambda^{-1+2\epsilon}.\]
Since the only critical point in is $\tau$ when
\[\rho=\Sigma_{1}(\bar{\xi},\eta)-\Sigma_{1}(0,\eta)\]
we can say that when $|\rho|\geq{}\lambda^{-1+3\epsilon}$ the phase is non-stationary and integration by parts yields an $O(\lambda^{-\infty})$ contribution. So 
\begin{multline*}T_{0}f=\frac{\lambda}{2\pi }\int e^{i\lambda\left(\langle \bar{y},\bar{\xi}+\langle z,\eta\rangle+t\Sigma_{1}(0,\eta)+(t-\tau)\rho+\tau(\Sigma_{1}(\bar{\xi},\eta)-\Sigma_{1}(0,\eta)\right)}\\
\times \chi_{0}(\nabla_{\bar{\xi}}\Sigma_{1}(\bar{\xi},\eta))a(\lambda^{1-3\epsilon}\rho)b(\bar{\xi},\eta)f(\bar{\xi},\eta)d\tau d\rho d\bar{\xi}d\eta+O(\lambda^{-\infty})\end{multline*}
We re-write this as
\[T_{0}f=I_{0}F(\bar{y},\cdot)\]
where
\[I_{0}g=\frac{\lambda^{\frac{1}{2}}}{2\pi}\int e^{i\lambda\left(\langle z,\eta\rangle+t\Sigma_{1}(0,\eta)\right)}a(\lambda^{1-3\epsilon} \rho)G(\rho,\tau,\eta)d\rho d\tau d\eta\]
\[F(\bar{y},\rho,\tau,\eta)=\lambda^{\frac{1}{2}}\int e^{i\lambda\left(-\tau\rho+\langle \bar{y},\bar{\xi}\rangle+\tau(\Sigma_{1}(\bar{\xi},\eta)-\Sigma_{1}(0,\eta))\right)}a\left(\frac{\lambda^{1-3\epsilon}\rho}{2}\right)f(\bar{\xi},\eta)d\bar{\xi}\]
We have already established that
\begin{equation}\norm{I_{0}g}_{L^{2}_{t}L^{p}_{z}}\lesssim \lambda^{-\frac{d-k-1}{2}+\frac{\epsilon}{2}}\log(\lambda)\norm{g}_{L^{2}_{\rho}L^{2}_{\tau}L^{2}_{\eta}}.\label{I0est}\end{equation}
So
\[\norm{T_{0}f}_{L^{2}_{\bar{y}}L^{2}_{t}L^{p}_{z}}\lesssim \lambda^{-\frac{d-k-1}{2}+\frac{\epsilon}{2}}\log(\lambda)\norm{F}_{L^{2}_{\bar{y}}L^{2}_{\rho}L^{2}_{\tau}L^{2}_{\eta}}.\]
Finally,
\[\norm{F}^{2}_{L^{2}_{\bar{y}}L^{2}_{\rho}L^{2}_{\tau}L^{2}_{\eta}}=\lambda\int e^{i\lambda\left\langle \bar{y},\bar{\xi}-\bar{\zeta}\rangle+\tau(\Sigma_{1}(\bar{\xi},\eta)-\Sigma_{1}(\bar{\zeta},\eta))\right)}f(\bar{\xi},\eta)a^{2}\left(\frac{\lambda^{1-3\epsilon}\rho}{2}\right)\bar{f}(\bar{\zeta},\eta)d\bar{y}d\rho d\tau \bar{\xi}d\bar{\zeta} d\eta.\]
The phase is only critical when $\bar{\xi}=\bar{\zeta}$ so an application of Schur's test (along with the small support for $\rho$) yields,
\[\norm{F}^{2}_{L^{2}_{\bar{y}}L^{2}_{\rho}L^{2}_{\tau}L^{2}_{\eta}}\lesssim \lambda^{-(k-1)+\frac{3\epsilon}{2}}\norm{f}_{L^{2}_{\bar{\xi}}L^{2}_{\eta}}.\]
Therefore
\[\norm{T_{0,\lambda}f}_{L^{2}_{\bar{y}}L^{2}_{t}L^{p}_{z}}\lesssim \lambda^{-\frac{d-2}{2}+2\epsilon}\log(\lambda)\norm{f}_{L^{2}_{\bar{\xi}}L^{2}_{\eta}}.\]

\end{proof}

\begin{prop}\label{prop:TSboundinf}
Suppose that for each $m$ the hypersurface $\zeta=\Sigma_{m}(0,\eta)$ sitting in $\R^{2}$ is curved. Then
    \begin{equation}
\norm{\prod_{m=1}^{d-1}\mathcal{E}_{m,\lambda}f_{m}}_{L^{\frac{4}{2d-3}}_{y}L^{\infty}_{z}}\lesssim\lambda^{-\frac{k(2d-3)}{4}+3\epsilon}\prod_{m=1}^{d-1}\norm{f_{m}}_{L^{2}(H_{m})}\end{equation}
\end{prop}

\begin{proof}
This proof is structured in the same fashion as that of Proposition \ref{prop:TSbound2} but requires a little more set up. First any fixed $(y_{1},\dots,y_{d-2})$ we can use a $d-1$ product version of the H\"{o}lder estimate to show that
\begin{align*}\norm{\prod_{m=1}^{d-1}\mathcal{E}_{m,\lambda}f_{m}(y_{1},\dots,y_{d-2},\cdot)}_{L^{\frac{4}{2d-3}}_{y_{d-1}}L^{\infty}_{z}}&\leq \norm{\mathcal{E}_{d-1,\lambda}f_{d-1}(y_{1},\dots,y_{d-2},\cdot)}_{L_{y_{d-1}}^{\frac{4r}{2d-3}}L_{z}^{\infty}}\\
&\times \prod_{m=1}^{d-2}\norm{\mathcal{E}_{m,\lambda}f_{m}(y_{1},\dots,y_{d-2},\cdot)}_{L_{y_{d-1}}^{\frac{4s}{2d-3}}L_{z}^{\infty}},\end{align*}
for any $(r,s)$ so that
\[\frac{1}{r}+\frac{d-2}{s}=1.\]
In this case we pick
\[r=2d-3\quad s=\frac{2d-3}{2}\]
to arrive at
\begin{align*}\norm{\prod_{m=1}^{d-1}\mathcal{E}_{m,\lambda}f_{m}(y_{1},\dots,y_{d-2},\cdot)}_{L^{\frac{4}{2d-3}}_{y_{d-1}}L^{\infty}_{z}}&\leq \norm{\mathcal{E}_{d-1,\lambda}f_{d-1}(y_{1},\dots,y_{d-2},\cdot)}_{L_{y_{d-1}}^{4}L_{z}^{\infty}}\\
&\times\prod_{m=1}^{d-2}\norm{\mathcal{E}_{m,\lambda}f_{m}(y_{1},\dots,y_{d-2},\cdot)}_{L_{y_{d-1}}^{2}L_{z}^{\infty}}.\end{align*}
We repeat the argument with the $y_{d-2}$ variable to obtain
\begin{align*}    
\norm{\prod_{m=1}^{d-1}\mathcal{E}_{m,\lambda}f_{m}(y_{1},\dots,y_{d-3},\cdot)}_{L^{\frac{4}{2d-3}}_{(y_{d-2},y_{d-1})}L^{\infty}_{z}}&\leq \norm{\mathcal{E}_{d-1,\lambda}f_{d-1}(y_{1},\dots,y_{d-3},\cdot)}_{L_{y_{d-2}}^{2}L_{y_{d-1}}^{4}L_{z}^{\infty}}\\
&\times \norm{\mathcal{E}_{d-2,\lambda}f_{d-2}(y_{1},\dots,y_{d-3},\cdot)}_{L^{4}_{y_{d-2}}L_{y_{d-1}}^{2}L_{z}^{\infty}}\\
&\times \prod_{m=1}^{d-3}\norm{\mathcal{E}_{m,\lambda}f_{m}(y_{1},\dots,y_{d-3},\cdot)}_{L_{(y_{d-2},y_{d-1})}^{2}L_{z}^{\infty}}.\end{align*}
After repeating the process $k$ times we arrive at
\[\norm{\prod_{m=1}^{d-1}\mathcal{E}_{m,\lambda}f_{m}}_{L^{\frac{4}{2d-3}}_{y}L^{\infty}_{z}}\leq \prod_{m=1}^{k}\norm{\mathcal{E}_{m,\lambda}f_{m}}_{L^{2}_{y_{1}}\cdots L^{4}_{y_{m}}L^{2}_{y_{m+1}}\cdots L^{2}_{y_{k}}L^{\infty}_{z}}.\]
It is therefore enough to prove bounds separately on each factor
\[\norm{\mathcal{E}_{m,\lambda}f_{m}}_{L^{2}_{y_{1}}\cdots L^{4}_{y_{m}}L^{2}_{y_{m+1}}\cdots L^{2}_{y_{k}}L^{\infty}_{z}}.\]
By the mixed Lebesgue space version of Minkowski's inequality we may permute the order of the iterated integrals so long as we move larger weightings to inner integrals \cite{F87}. Therefore it is enough to control
\begin{equation}\norm{\mathcal{E}_{1,\lambda}f_{1}}_{L^{2}_{\bar{y}}L^{4}_{y_{1}}L^{\infty}_{z}}\label{mixedcontrol}\end{equation}
where we have again written $\bar{y}=(y_{2},\dots,y_{d-1})$. We now proceed in a similar fashion to the proof of Proposition \ref{prop:TSbound2}. The same argument reduces controlling \eqref{mixedcontrol} to controlling the $L^{2}_{\eta}\to L^{4}_{y_{1}}L^{\infty}_{z}$ mapping norm of
\[I_{r}(\tau,t)f=\int e^{i\lambda\left(\langle z,\eta\rangle +t\Sigma_{1}(0,\eta)\right)+\lambda^{\epsilon (r+1)}t\rho}a(\rho)a(\lambda^{\epsilon (r-1)}(t-\tau))f(\rho,\tau,\eta)d\rho d\eta.\]
As before we obtain the $L^{1}\to L^{\infty}$ estimate by writing 
\[I_{r}(\tau,t)I^{\star}(\tau,s)g=\int e^{i\lambda\left( \langle z-w,\eta\rangle+t\Sigma_{1}(0,\eta)\right)+\lambda^{\epsilon r}(t-s)\rho}a^{2}(\rho)g(w)d\eta dw\]
and computing  $\eta$ integral via stationary phase. We have
\[\norm{I_{r}(\tau,t)I_{r}(\tau,t)}_{L^{1}_{w}\to L^{\infty}_{z}}\lesssim \lambda^{-\frac{1}{2}}(\lambda^{-1}+|t-s|)^{-\frac{1}{2}}.\]
We then resolve the time integral via Hardy-Littlewood-Sobolev. Unsurprisingly the power $-\frac{1}{2}$ is exactly the critical one for an $L^{4}$ norm. Then 
\[\norm{I_{r}(\tau,\cdot)}_{L^{2}_{\rho}L^{2}_{\eta}\to L^{4}_{t}L^{\infty}_{z}}\lesssim \lambda^{-\frac{1}{4}}\]
and so
\[\norm{\mathcal{E}_{1,\lambda}}_{L^{2}_{\bar{\xi}}L^{2}_{\eta}\to L^{2}_{\bar{y}}L^{4}_{y_{d-1}}L^{\infty}_{z}}\lesssim \lambda^{-\frac{2d-3}{4}+3\epsilon}.\]
    \end{proof}

Finally we can state the full $k$-linear theorem.

\begin{thm}\label{thm:mixedtranscurv}
   Let $\lambda\geq{}1$ and $\epsilon>0$. Suppose $H_{1},\dots,H_{k}$ are a set of hypersurfaces in $\R^{d}$ whose normals $\nu_{m}$ obey a uniform transversality condition
   \[|\nu_{1}(\xi^{1})\wedge \cdots \wedge \nu_{k}(\xi^{m})|\geq{}c.\]
   For $\boldsymbol{\xi}=(\xi^{1},\dots,\xi^{k})$ with $\xi^{m}\in\mathcal{H}_{m}$ denote the subspace spanned by the  $\nu_{m}(\xi^{m})$ as $\mathcal{Z}_{\boldsymbol{\xi}}$ and $\gamma_{\mathcal{Z^{\perp}_{\boldsymbol{\xi}}}}$ the projection onto $\mathcal{Z^{\perp}_{\boldsymbol{\xi}}}$.  Then if for each $m$ the  shape operator $S_{m}$ of $H_{m}$ obeys the curvature conditions that $\gamma^{\star}_{\mathcal{Z}_{\boldsymbol{\xi}}^{\perp}}S_{m}\gamma_{\mathcal{Z}_{\boldsymbol{\xi}}^{\perp}}$ is non-degenerate,
   \[\norm{\prod_{m=1}^{k}\mathcal{E}_{m,\lambda}f_{m}}_{L^{\frac{p(k)}{k}}}\lesssim \lambda^{-\frac{p(k)}{k}+\epsilon}\prod_{m=1}^{k}\norm{f_{m}}_{L^{2}}\quad p(k)=\frac{2(d+k)}{d-k-2}\]
   
\end{thm}

\begin{proof}
The geometric assumption on $\gamma^{\star}_{\mathcal{Z}_{\boldsymbol{\xi}}^{\perp}}S\gamma_{\mathcal{Z}_{\boldsymbol{\xi}}^{\perp}}$ ensures that for any $\bar{\xi}$ the section $\zeta=\Sigma_{m}(\bar{\xi},\eta)$ is curved when seen as a surface in $\R^{d-k+1}$, in particular when $\bar{\xi}=0$. So the results of Propositions \ref{prop:TSbound2} and \ref{prop:TSboundinf} hold. Therefore for $k\leq d-2$ interpolate between the results of \ref{prop:klintrans} and \ref{prop:TSbound2} and for $k=d-1$ interpolate between the results of \ref{prop:klintrans} and \ref{prop:TSboundinf}.
\end{proof}

\begin{remark}
    A natural question now is; could the $\lambda^{\epsilon}$ loss be removed? For $p>p(k)$ this is likely to be possible. In the proofs of Propositions \ref{prop:TSbound2} and \ref{prop:TSboundinf} we gave up factors of $\lambda^{\epsilon}$, but this was largely for convenience. It is likely that those arguments could be replaced with a dyadic cut off around critical points. The $\lambda^{\epsilon}$ loss in Proposition \ref{prop:klintrans} is more difficult to remove, it comes from Theorem \ref{thm:simsat}. So to remove the $\lambda^{\epsilon}$ factor there we would need to produce a new version of the proof of Theorem \ref{thm:simsat} that allowed us to work with points that are $\lambda^{-1}$ separated, rather than $\lambda^{-1+\epsilon}$. For such points we would be unable to get $O(\lambda^{-\infty})$ decay if $p\CC_{m}q$ does not hold but rather would get decay in the form $(1+\lambda^{-1}|p-q|)^{-R}$ for any $R\in\N$. Therefore when analysing the loops we would not be able to make the simple division into trivial loops, for which we accept the trivial bounds  and non-trivial loops, which make negligible contribution. We would need to do a finer analysis about how many non-trivial links in the loops, gaining decay as we look at loops with more non-trivial links. Since the number of points $q$ at a fixed distance from $p$ grows at a fixed polynomial rate and the decay rate is better than any polynomial it is likely that this sort of strategy would work. However it would add significant complication to the proof of Theorem \ref{thm:simsat} so we do not pursue it here.
\end{remark}

\bibliography{references}
\bibliographystyle{plain}

\end{document}